\documentclass[reqno,oneside,a4paper,11pt]{amsart}
\usepackage[utf8]{inputenc}
\usepackage{mathtools} 

\usepackage{amsmath,amsthm,amsbsy,amssymb, comment, booktabs, float}
\usepackage{arydshln} 
\usepackage{transparent} 
\usepackage{physics} 

\usepackage{xparse}     

\NewDocumentCommand{\mybar}{ O{0.860} O{0.3pt} m }{
    \mathrlap{\hspace{#2}\overline{\scalebox{#1}[1]{\phantom{\ensuremath{#3}}}}}\ensuremath{#3}
}

\newcommand{\cF}{\mathcal{F}}
\newcommand{\cG}{\mathcal{G}}

\newcommand{\cK}{\mathcal{K}}
\newcommand{\cL}{\mathcal{L}}

\newcommand{\cM}{\mathcal{M}}

\newcommand{\cP}{\mathcal{P}}
\newcommand{\cR}{\mathcal{R}}

\usepackage{mathrsfs}

\newcommand{\scP}{\mathscr{P}}
\newcommand{\scR}{\mathscr{R}}

\makeatletter
\renewcommand{\pod}[1]{\allowbreak\mathchoice
  {\if@display \mkern 18mu\else \mkern 2mu\fi (#1)}
  {\if@display \mkern 6mu\else \mkern 2mu\fi (#1)}
  {\mkern4mu(#1)}
  {\mkern4mu(#1)}
}
\makeatother

\usepackage{bm}

\newcommand{\legendre}[2][p]{\ensuremath{\left( \frac{#2}{#1} \right) }}

\newcommand{\sdfrac}[2]{\mbox{\small$\displaystyle\frac{#1}{#2}$}}

\newcommand{\length}{\operatorname{length}}
\newcommand{\RmK}{\cR_K^{{|\cdot|}}}

\newcommand{\PG}{\operatorname{PG}}
\newcommand{\dM}{\operatorname{d_M}}
\newcommand{\dE}{\operatorname{d_E}}
\newcommand{\dP}{\operatorname{d_{pc}}}
\newcommand{\scBpc}{\mathscr{B}_{\mathrm{pc}}}

\newcommand{\dens}{\operatorname{\delta}}

\newcommand{\bigabs}[1]{\big\lvert #1 \big\rvert}

\newcommand{\ZZ}{\mathbb{Z}}

\theoremstyle{plain}
\newtheorem{theorem}{Theorem}

\newtheorem{lemma}{Lemma}[section]
\newtheorem{proposition}{Proposition}[section]

\newtheorem{conjecture}{Conjecture}

\theoremstyle{remark}

\theoremstyle{definition}

\usepackage[font={up,small}]{caption}
\usepackage{xcolor} 
\definecolor{orange}{rgb}{1,0.5,0}
\definecolor{Ggreen}{rgb}{0.,0.575,0.0128}
\definecolor{Bblue}{rgb}{0.016,.132,0.91}
\usepackage[hyphens]{url}

 \usepackage[backref=page]{hyperref}
\hypersetup{
    colorlinks = true,
linkcolor={red},
urlcolor={Bblue},
citecolor={Ggreen},
urlcolor = {Bblue},
citebordercolor = {0.33 .58 0.33},
 linkbordercolor = {0.99 .28 0.23},
 breaklinks=true
}
\usepackage{cite}

\usepackage{subfigure}  
\usepackage[pdftex]{graphicx}
\graphicspath{{IMAGES/}{IMAGES1/}{./}}

\makeatletter
\def\mysequence#1{\expandafter\@mysequence\csname c@#1\endcsname}
\def\@mysequence#1{%
  \ifcase#1\or left\or right\or altceva\else\@ctrerr\fi}
\makeatother

\usepackage{tikz}
\usetikzlibrary{arrows.meta,    
                bending,        
                tikzmark}

\makeatletter
\@namedef{subjclassname@2020}{\textup{2020} Mathematics Subject Classification}
\makeatother
\begin{document}

\title[On the trajectories of a particle in a translation invariant involutive field]
{On the trajectories of a particle in a translation invariant involutive field}



\author[Cristian Cobeli]{Cristian Cobeli}
\address[Cristian Cobeli]{``Simion Stoilow'' Institute of Mathematics of the Romanian Academy,~21 Calea Grivitei Street, P. O. Box 1-764, Bucharest 014700, Romania}
\email{cristian.cobeli@imar.ro}


\author{Alexandru Zaharescu}
\address[Alexandru Zaharescu]{Department of Mathematics, 
University of Illinois at Urbana-Champaign, 
1409 West Green Street, Urbana, IL 61801, USA,
%
and 
``Simion Stoilow'' Institute of Mathematics of the Romanian Academy,~21 
Calea Grivitei 
Street, P. O. Box 1-764, Bucharest 014700, Romania}
\email{zaharesc@illinois.edu}

\subjclass[2020]{Primary 11B37; Secondary 11B50}


\thanks{Key words and phrases: 
lattice points,
partition with parabolas,
modular prime covering, 
discrete trajectory,
parabolic-taxicab distance,
translation-invariant-involutive operator}

\begin{abstract}
We introduce a double-folded operator that,  upon iterative application,  generates a dynamical system with two types of trajectories: a cyclic one and, another that grows endlessly on parabolas. These trajectories produce two distinct partitions of the set of lattice points in the plane.

Our object is to analyze these trajectories and to point out
a few special arithmetic properties of the integers they represent.

We also introduce and study the parabolic-taxicab distance, which measures
the fast traveling on the  steps of the stairs defined by points on the parabolic trajectories whose coordinates are based on triangular numbers.
\end{abstract}
\maketitle


\vspace{-4mm}
\section{Introduction}

The iteration of a double-folded $3$-dimensional operator has led to an engaging 
partitioning of the $\ZZ\times\ZZ$ plane with L\"oschian numbers
(see~\cite[{Chap. 10}]{Loc1940}, \cite{KRNG2024}, \cite[{\href{https://oeis.org/A003136}{A003136}}]{oeis})
on a lozenge configuration.
Besides being simple geometric objects, lozenges appear in various contexts, of which some are related to complex combinatorial counting problems (see~\cite{CF2023,CL2019, Ciucu2009, Ciucu2005}), while others seemingly simpler like the modified version of the
$\sqrt{2} : 1$ ratio rhombus  that lead to the wonderful stable three-dimensional 
\textit{triamon}-\textit{bamboozle} structure with multidirectional tubes~\cite{VV2013}.
In a related manner, in this paper we will lower the dimension by defining a pair of two-dimensional operators to which we will impose the same requirements as in~\cite{BCZ2023b} to be involutions and 
invariant under translations.
There are two types of such operators, which depend on the choice of two parameters.
We will see the interesting fact that these operators play a role analogous to the thresholds set by the 
orbital speed and the escape velocity in the movement of a particle in a gravitational field.
Thus, starting with a generic point and applying iteratively the first pair of operators 
results in a sequence of points that cycle on a satellite trajectory, 
while the repeated application of the second pair of operators generates 
a sequence of points confined on a parabolic trajectory.
The union of all these disjoint trajectories generates
two partitions of the set of lattice points $\ZZ\times\ZZ$
with twisted cycles in the first case and with discrete parabolas in the second.
\subsection{Definitions and statements of the main results}
Let $\alpha$ and $ \beta$ be fixed integers and consider the two folded
operator $\cF :=\{F',F''\}$, where
\begin{align}\label{eqF}
    F'(x,y)& = (\alpha x+\beta y+1,y),\qquad
    F''(x,y) = (x,\beta x+\alpha y+1).
\end{align}
The iteration of this operator, starting with different initial values, 
produces discrete dynamical systems that are essentially dependent 
on the parameter values $\alpha$ and $ \beta$.

Comparable procedures employed in closely related contexts also appear 
in other recently studied problems analyzed from various perspectives 
such as those that discuss 
the properties of numbers and patterns that appear in Pascal-like 
triangles~\cite{CZ2013, Pru2022},
the systems of numbers generated by Ducci-game-rules~\cite{CCZ2000, CPZ2016, CZ2014},
or the distribution of higher-order differences obtained through the 
iterative application of the $\PG$ operator in relation to the Proth-Gilbreath 
conjecture~\cite{BCZ2023, CZZ2013, Gil2011, Guy1988, Guy2004, Pro1878}.

Two properties that we require for operators $F'$ and 
$F''$ defined by~\eqref{eqF} are \textit{translation invariance} 
and the property of \textit{involution} (composition with itself equals the identity), 
as they yield intriguing trajectories from geometric and arithmetic perspectives.
%
Accordingly, we introduce two 
operators 
$\cK :=\{K',K''\}$ and
$\cL :=\{L',L''\}$, which are defined by
\begin{align}\label{eqK}
    K'(x,y)& = (-x+y+1,y),& K''(x,y)& = (x,x-y+1),\\
    L'(x,y)& = (-x+2y+1,y),& L''(x,y)& = (x,2x-y+1).
    \label{eqL}
\end{align}
Our object is to study the sequences of points formed by iterating 
these operators, the arithmetic and geometric properties of 
the patterns they generate, and the characteristics of the integers 
represented throughout the process.

Let $\scP_*(a,b)$ denote the set of all pairs obtained through iterations starting from the initial pair $(a,b)$, thus, in particular,
\begin{equation}\label{eqPKL}
    \scP_K(a,b) := \bigcup_{n\ge 0} K^{[n]}(a,b)\quad
    \text{ and }\quad
    \scP_L(a,b) := \bigcup_{n\ge 0} L^{[n]}(a,b)\,.
\end{equation}
(The superscript notation indicates the composition of $n$ operators of 
any type, with~a single prime or a double prime, taken from $\cF$, 
where, in particular, $\cF$ can be either $\cK$ or $\cL$.)

Denote the sets of \textit{represented integers} by
\begin{equation*}
   \begin{split}
    \cR_K(a,b) &:= \big\{ m \in \{x,y\} :  (x,y) \in\cP_K(a,b)\big\}\,,\\
    \cR_L(a,b) &:= \big\{ m \in \{x,y\} :  (x,y) \in\cP_L(a,b)\big\}\,.
   \end{split} 
\end{equation*}

\medskip

We remark that all four operators $K',K'',L',L''$ are involutions
and the operators $L'$ and~$L''$ are invariant under translations
    $(x,y)\mapsto (x,y)+h:=(x+h,y+h)$, while $K'$ and $K''$ are not.
A few graphical representations of the $K$-generated cycles and the paths they induce are shown in Figures~\ref{FigureCycles0M},
\ref{FigureCycles-MM},
\ref{FigureCycles},
\ref{FigurePaths}.
\begin{figure}[hbt]
 \centering
  \includegraphics[width=0.19\textwidth]{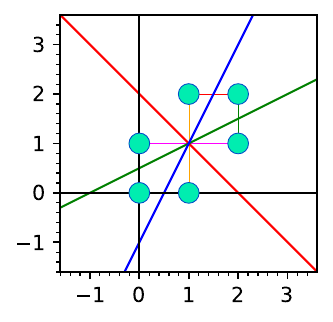}
  \includegraphics[width=0.19\textwidth]{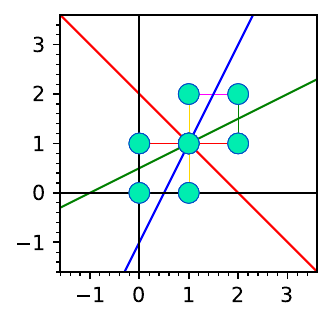}
  \includegraphics[width=0.19\textwidth]{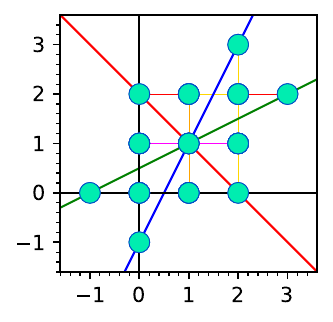}
  \includegraphics[width=0.19\textwidth]{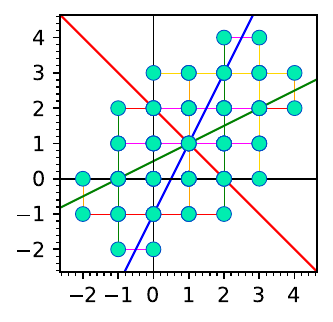}
  \includegraphics[width=0.19\textwidth]{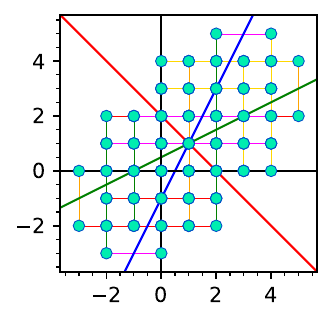}
\caption{
The smallest $K$-generated cycles by $(a,b)\in[0,T]^2$, with $T=0,1,2,3,4$.
}
 \label{FigureCycles0M}
 \end{figure}

The next theorem describes the cycles generated by the operators in $\cK$.
\begin{theorem}\label{TheoremK}
  Let $a$ and $b$ be integers. Then the following hold.
\begin{itemize}
\item[(1)]
Starting with $(a,b)$ and applying the operators in 
$\cK=\{K', K''\}$ alternatively creates a cycle with six points.
If any of these points is located on the lines 
$y = -x+2$,  
$2y = x+1$,  
$y = 2x-1$,
then some of the cycle points overlap, forming a cycle with only three distinct points, unless $a=b=1$, in which case all six points coincide.

\item[(2)]
The union of all cycles forms a partition of the set of lattice points in the plane. The average Euclidean length of the paths connecting the six points of the cycles generated by the points $(a,b)$ with $-T\le a,b\le T$ is $\frac{17}{3}T+O(1)$.

\item[(3)]
There are infinitely many cycles whose coordinates represent
three squares or three cubes.
There exist cycles whose coordinates in absolute value are all prime numbers.
\end{itemize}
\end{theorem}

\begin{figure}[ht]
 \centering
  \includegraphics[width=0.19\textwidth]{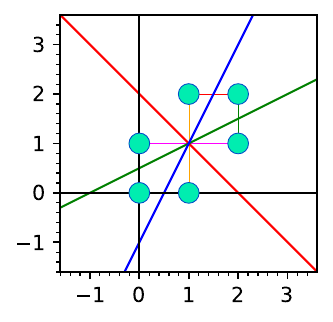}
  \includegraphics[width=0.19\textwidth]{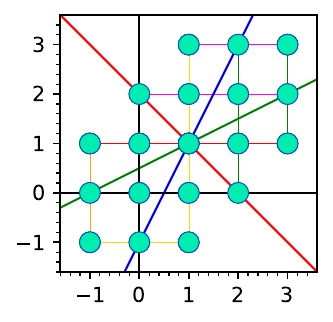}
  \includegraphics[width=0.19\textwidth]{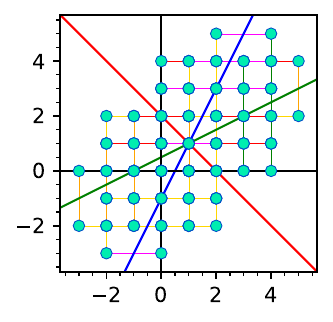}
  \includegraphics[width=0.19\textwidth]{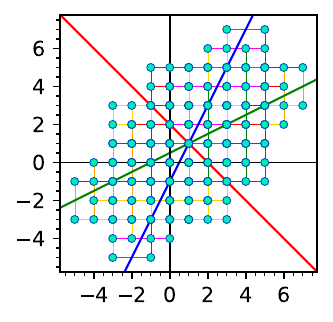}
  \includegraphics[width=0.19\textwidth]{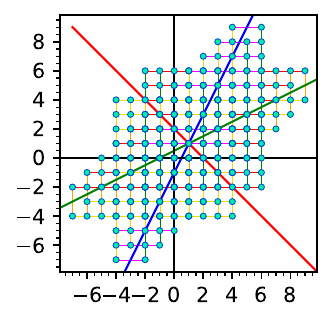}
\caption{
The smallest $K$-generated cycles by $(a,b)\in[-T,T]^2$, with $T=0,1,2,3,4$.
}
 \label{FigureCycles-MM}
 \end{figure}

We say that two pairs of integers belonging to the set 
$\scP_F(a,b)$ are $F$-$(a,b)$-\textit{equivalent}. 
We also refer to them as simply \textit{equivalent}
if the meaning is clear from the context.
As a result, we obtain two partitions of the set of lattice points 
from the plane, one corresponding to operators from~$\cK$ 
and the other to operators from~$\cL$.

In addition, as in the case of the three-dimensional analog operators~\cite{BCZ2023b}, 
we will use the same word to say that two such sets
$\scP_F(a,b)$
consisting of equivalent pairs 
are themselves \textit{equivalent} to each other if one can be obtained 
from the other through a translation.

We mention that in the case of the analogue $3$-dimensional operators, 
there are four towers of equivalence classes of triples, 
of which three are similar, obtained from each other 
by a rotation~\cite[Theorems 3]{BCZ2023b}.
One would like to know how many equivalence classes are there and also 
how they look like in the $2$-dimensional case.
The following theorem answers this question and also provides a 
complete and explicit description of them.

\begin{theorem}\label{TheoremMain}
   The set of pairs $\scP_{L}(a,b)$ defined by~\eqref{eqPKL} form
   a partition of the set of lattice points of the plane as follows.
\begin{enumerate}
    \item 
Each set $\scP_{L}(a,b)$ has a first element $(m,m)$
with the property that $m\le \min\{x,y\}$ for any 
$(x,y)\in \scP_{L}(a,b)$.
    \item 
The coordinates of any point in $\scP_{L}(0,0)$, 
other than the origin, are $(T_k,T_{k+1})$
or $(T_{k+1}, T_k)$, for $k\ge 0$,
where $T_k=k(k+1)/2$ is the $k$-th triangular 
number.
    \item
The points in $\scP_{L}(0,0)$ lie on the parabola of equation 
$x+y = (x-y)^2$.
    \item 
Any set $\scP_{L}(a,b)$ is equivalent modulo a translation 
with $\scP_{L}(0,0)$.
    \item 
The lattice point $(a,b)$ belongs to the parabola 
$\scP_{L}(a,b)$ whose vertex is $(m,m)$,
where $m=\frac{1}{2}(a+b-(a-b)^2)$.
\end{enumerate}
\end{theorem}
The parabola on which the points of $\scP(0,0)$ lie is shown in Figure~\ref{FigureLadders}. 
In Figure~\ref{Figure17parabolas}, all the parabolas described in Theorem~\ref{TheoremMain} are shown
partitioning the plane as follows 
\begin{equation}\label{eqPartition}
    \ZZ^2 = \bigcup_{m\in\ZZ} \scP_L(m,m).
\end{equation}

As an immediate follow up of this partition, let us consider 
a particle moving swiftly up and down 
the ladders in Figure~\ref{FigureLadders}, 
from any lattice point on a parabola 
to another point on~it.
Additionally, to enable the particle to move
to any location on the plane, it should have the option to ``jump''
from one parabolic trajectory to another,
moving along the square grid in small steps, from a lattice point to any 
of its four neighboring points.
Thus, we obtain a distance function, which we call \textit{parabolic-taxicab} 
distance and denote by $\dP$.
In this manner, assuming $P$ and $Q$ are two points of integer coordinates,
$\dP(P,Q)$ is defined as the smallest number of steps required to travel
from $P$ to $Q$ in successive steps, which can be either climbing 
up or down the ladders of a parabola $\cP_L(m,m)$, for some $m\in\ZZ$, 
or making unit jumps vertically or horizontally on the square grid. 
(For example, in Figure~\ref{FigureTwoPaths}, two such paths are shown, with one being minimal.)
As a consequence, we obtain a special looking ball (see Figure~\ref{FigureBalls})
whose exact formula for the area is likely to be the one provided next,
if the center is on the first diagonal.

\begin{figure}[ht]
 \centering
 \mbox{
 \subfigure{
    \includegraphics[width=0.49\textwidth]{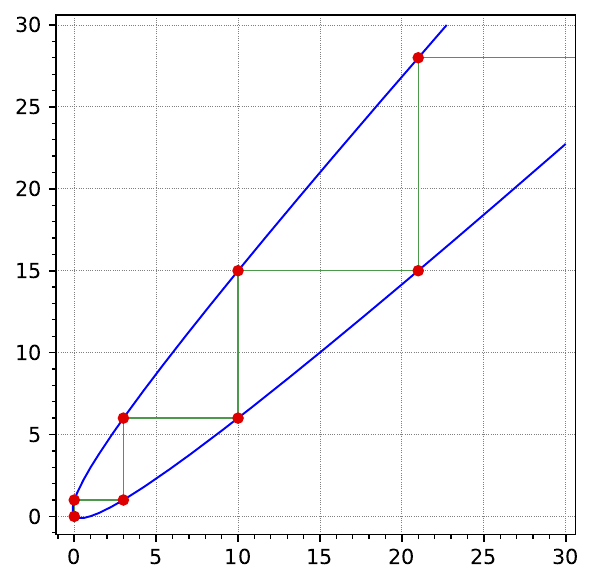}
 \label{FigureDotsOnParabolaK1K2}
 }
 \subfigure{
    \includegraphics[width=0.49\textwidth]{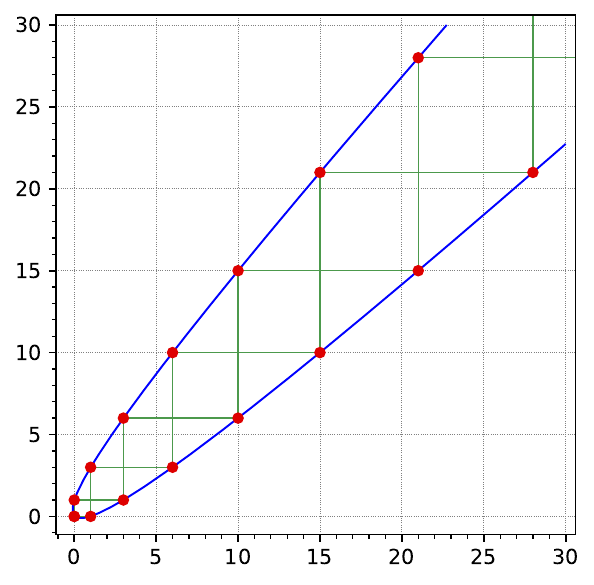}
 \label{FigureBothLadders}
 }
 }
\caption{
The lattice points in $\scP_L(0,0)$, starting with the vertex
$(0,0)$.
The points on the upper branch have coordinates $(T_k,T_{k+1})$,
while those on the lower branch have coordinates 
$(T_{k+1},T_k)$, where $T_k = k(k+1)/2$, for $k\ge 0$.
}
 \label{FigureLadders}
 \end{figure}
\begin{conjecture}\label{ConjectureBallArea}
 Let $m$ and $r\ge 0$ be integers.
Then, in the parabolic-taxicab geometry, 
 the measure of the closed ball of center $(m,m)$ and radius 
 $r$ is
\begin{equation}\label{eqAreaBall}
\mu\Big(\scBpc\big((m,m),r\big)\Big) = 
  \begin{cases}
    \frac{1}{12}\big(10 r^3+9r^2+26r+12\big), & 
   \text{if $r$ is even;} \\[4mm]
    \frac{1}{12}\big(10 r^3+9r^2+26r+15\big), & 
   \text{if $r$ is odd.}
  \end{cases}  
\end{equation}
\end{conjecture}
Precise definitions, examples, and other relevant information
about the parabolic-taxicab distance are discussed in
Section~\ref{SectionDistance}.
\begin{figure}[htb]
 \centering
 \subfigure{
    \includegraphics[width=0.48\textwidth]{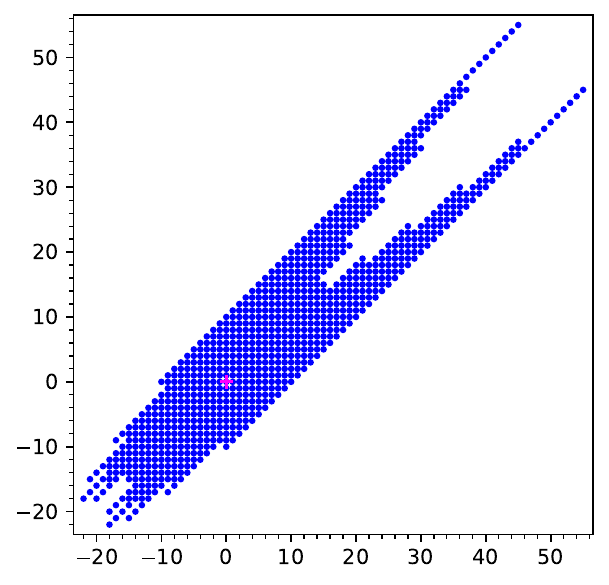}\
    \includegraphics[width=0.48\textwidth]{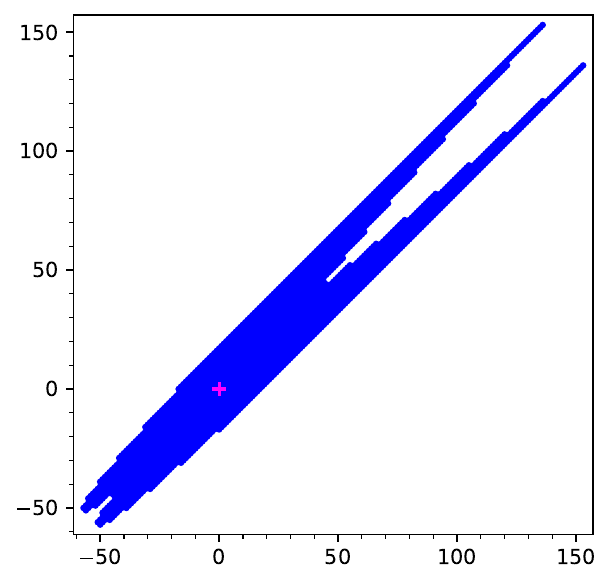}
 }
\caption{
The balls $\scBpc(O,r)$ centered at the origin $O=(0,0)$
and radii $r=10, 17$. 
The areas of the balls are $931$ and $4349$, respectively.
}
 \label{FigureBalls}
 \end{figure}

Next we study the distribution of the represented integers 
$\cR(a,b)$ in residue classes.
Since, according to Theorem~\ref{TheoremMain}, any parabola $\cP_L(a,b)$
is obtained as a translation along the first diagonal of the parabola
$\cP_L(0,0)$, it is enough to analyse the later case.

Let $T>0$ and denote by $N_L(l,p;T)$ the number of represented integers less than~$T$ that belong to the residue class $l\pmod p$, that is,
\begin{equation}\label{eqN}
   N_L(l,p;T) := \#\big\{
   r\in\scR_L(0,0) : r\le T,\ r\equiv l \pmod p
   \big\}.
\end{equation}
Then, the required densities will be given by
\begin{equation}\label{eqLimitDensity}
   \dens_L(l,p) := \lim\limits_{T\to\infty}
   \sdfrac{1}{\#\big(\scR_L(0,0)\cap[0,T]\big)}
   N_L(l,p;T),
\end{equation}
for $0\le l< p$.

\begin{theorem}\label{TheoremLmodp}
    Let $p$ be prime and let $\dens_L(l,p)$ denote the limit density of the residue class
$l\equiv R \mod p$ of the represented integers 
$R \in \scR_L(0, 0)$, for $l = 0, 1,\dots, p-1$. Then:
\begin{enumerate}
    \item If $p=2$, then $\dens_2(0)=\dens_2(1)=\frac{1}{2}$.
    \item If $p\ge 3$, then
    \begin{equation}\label{eqdensLpl}
    \begin{cases}
  \dens_L(l,p) = 0, & \text{ if } \legendre[p]{2l+2^{-2}}=-1; \\[2mm]
  \dens_L(l,p) =\frac{1}{p},  & 
  \text{ if } l\equiv -2^{-3}\pmod p;\\[2mm]
  \dens_L(l,p) =\frac{2}{p}, & 
  \text{ if } \legendre[p]{2l+2^{-2}}=1.
    \end{cases}
\end{equation}
\end{enumerate}
(Here $\legendre[p]{a}$ denotes the Legendre symbol and the inverses are taken modulo $p$.)
\end{theorem}

\begin{figure}[htb]
 \centering
    \includegraphics[width=0.50\textwidth]
    {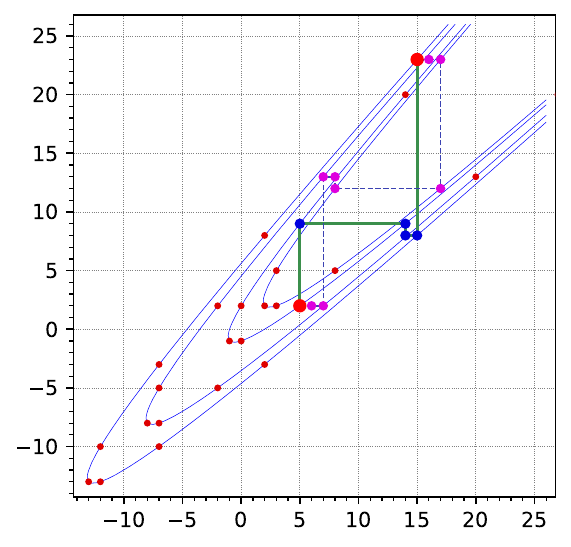}
    \hfill
        \includegraphics[width=0.49\textwidth]
    {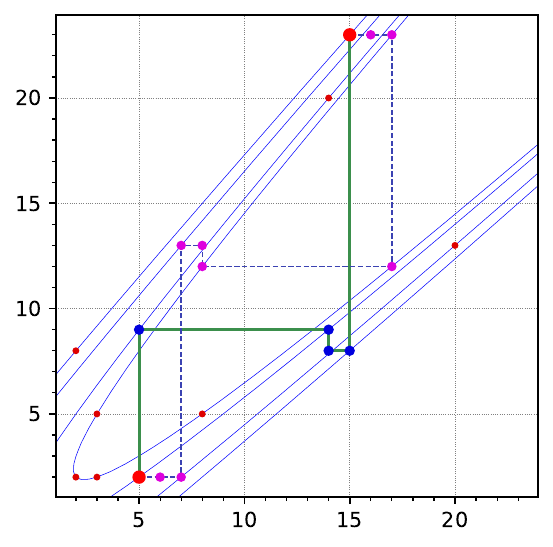}
\caption{Two paths connecting $P=(5,2)$ and $Q = (15,23)$ through segments that combine taxicab steps with steps climbing the parabolic ladders shown in two zoom-levels.
The first path, proving to be minimal, passes through points 
$(5, 2)$, $(5, 9)$, $(14, 9)$, $(14, 8)$, $(15, 8)$, $(15, 23)$, 
and thus 
\mbox{$\dP(P,Q)=5$}.
The second path, shown dotted, has $9$ steps and passes through points
$(5, 2)$, $(6, 2)$, $(7, 2)$, $(7, 13)$, $(8, 13)$, $(8, 12)$,
$(17, 12)$, $(17, 23)$, $(16, 23)$, $(15, 23)$.
Compare \mbox{$\dP(P,Q)=5$} with the taxicab distance,
which is equal to $(15-5)+(23-2)=31$, and with the Euclidean distance, 
which is $\sqrt{541}\approx 23.25941$.
}
 \label{FigureTwoPaths}
 \end{figure}

\medskip
The paper is organized as follows. In Section~\ref{SectionLemmas}
we prove the basic properties of operators $K$ and $L$.
In Section~\ref{SectionKTrajectories} we study the cycles
generated by $K$, prove Theorem~\ref{TheoremK},
and give parameterized families of distinguished trajectories.
In Section~\ref{SectionProofThMain} we begin the study of 
the parabolic trajectories proving Theorem~\ref{TheoremMain}, and then continue in Section~\ref{SectionLDensities}, where we prove Theorem~\ref{TheoremLmodp}.
The parabolic-taxicab distance is discussed in Section~\ref{SectionDistance}.

\section{Preparatory Lemmas}\label{SectionLemmas}
\begin{lemma}\label{LemmaInvolutionsGeneral}
   The operators defined by~\eqref{eqF} 
have the following properties:
\begin{enumerate}
    \item $F'$ and $F''$ are involutions
if and only if either $\alpha = -1$ or 
they are equal to the identity.
    \item 
    $F'$ and $F''$ are invariant under translations
if and only if $\alpha +\beta = 1$.
    \item  
    $F'$ and $F''$ are at the same time both involutions and invariant to translations 
if and only if either $\alpha = -1$ and $\beta = 2$ or $F'$ and $F''$ are 
equal to the identity.
\end{enumerate} 
\end{lemma}
\begin{proof}
Due to the symmetry of the definitions, 
it is enough to prove the lemma for just one of the operators $F'$ or $F''$.

Checking for the involution property,
we note that
\begin{equation*}
  F'\big(F'(x,y)\big)=F'(\alpha x +\beta y,y)
=(\alpha(\alpha x+\beta y) +\beta y,y) 
= (\alpha^2 x + \beta(\alpha+1) y,y).
\end{equation*}
Then, the condition for $F'$ to be an
involution is that
$\alpha^2 x + \beta(\alpha+1) y = x$
for all integers $x$ and $y$, 
which occurs only if both equalities
$\alpha^2=1$ and $\beta(\alpha+1)=0$ hold.
Therefore $F'$ is an involution if either $\alpha = -1$ or $\alpha = 1$ and $\beta =0$, that is,
$\alpha = -1$ or $F'=Id$.

\smallskip
The condition for $F'$ to be invariant under translations is
\mbox{$F'(x+h,y+h)=F'(x,y)+h$} for all integers $x,y,h$.
Since the equality holds on the second coordinate, 
this is equivalent with
\begin{equation*}
  \alpha (x+h) +\beta (y+h) +1 = \alpha x+\beta y +1+h, \ \ \text{ for all $x,y,h\in\ZZ$},
\end{equation*}
which holds if and only if $\beta = 1-\alpha$.
This concludes the proof of the lemma.
\end{proof}

\begin{lemma}\label{LemmaLL}
Let $a,b$ be integers and, for any integer $n\ge 0$, let $K^{[n]}$ denote the composition of $n$
operators $K'$ and $K''$ applied alternatively 
and starting with $K'$,
that is, $K^{[0]}=Id$, $K^{[1]}=K'$,
$K^{[2]}=K''\circ K'$, 
$K^{[3]}=K'\circ K''\circ K'$,
$K^{[4]}=K''\circ K'\circ K''\circ K'$,
and so on. 
Then
\begin{equation*}
\begin{split}
  K^{[n]}(a,0) &= \left(
\Big(-2\Big\lfloor\frac{n+1}{2}\Big\rfloor+1\Big)a 
 + T_{2\lfloor \frac{n+1}{2}\rfloor -1},
 -2\Big\lfloor\frac{n}{2}\Big\rfloor  a
 + T_{2\lfloor \frac{n}{2}\rfloor}
 \right), \\ 
   K^{[n]}(0,b) &= \left(
2\Big\lfloor\frac{n+1}{2}\Big\rfloor b
 + T_{2\lfloor \frac{n+1}{2}\rfloor -1},
 \Big(2\Big\lfloor\frac{n}{2}\Big\rfloor +1\Big) b
 + T_{2\lfloor \frac{n}{2}\rfloor}
 \right), \\ 
 \end{split}
\end{equation*}
where $T_n=n(n+1)/2$ is the $n$th triangular number. 
\end{lemma}
 \begin{proof}
A direct calculation gives:
 \begin{align*}
   K^{[1]}(a,0)&=(-a + 1, 0)          & K^{[1]}(0,b)&=(2b + 1, b) \\
   K^{[2]}(a,0)&=(-a + 1, -2a + 3)    & K^{[2]}(0,b)&=(2b + 1, 3b + 3)\\ 
   K^{[3]}(a,0)&=(-3a + 6, -2a + 3)   & K^{[3]}(0,b)&=(4b + 6, 3b + 3) \\
   K^{[4]}(a,0)&=(-3a + 6, -4a + 10)  & K^{[4]}(0,b)&= (4b + 6, 5b + 10)\\ 
   K^{[5]}(a,0)&=(-5a + 15, -4a + 10) & K^{[5]}(0,b)&=(6b + 15, 5b + 10) \\
   K^{[6]}(a,0)&=(-5a + 15, -6a + 21) & K^{[6]}(0,b)&=(6b + 15, 7b + 21)
 \end{align*}
From these, the general formulas can be deduced as linear 
functions in $a$ or $b$ with coefficients depending on 
even/odd numbers and triangular numbers of even/odd order.
  
Then the result follows by induction.
 \end{proof}

\section{
\texorpdfstring{Arithmetic properties of the $K$-generated trajectories}{Arithmetic properties of K-generated trajectories}
}\label{SectionKTrajectories}

In this section, we look at the sequences of pairs of integers generated by the iterative application of the operators $K'$ and $K''$.

\subsection{The Cycles}
First, let us note that a sequence becomes stationary upon repeated application of either of 
the operators $K'$ and $K''$ since, according to Lemma~\ref{LemmaInvolutionsGeneral}, 
both of them are involutions.

Starting with the pair $(a,b)$, we obtain the sequence of points
\begin{equation}\label{eqSeqab}
\begin{split}
    (x,y) \xleftrightarrow{\text{ $K'$ }} (-x+y+1,y)
      \xleftrightarrow{\text{ $K''$ }} & (-x+y+1,-x+2)
      \xleftrightarrow{\text{ $K'$ }} (-y+2,-x+2)
      \xleftarrow{\text{ $K''$ }}\\
       \xrightarrow{\text{ $K''$ }}& (-y+2,x-y+1)
       \xleftrightarrow{\text{ $K'$ }} (x,x-y+1)
       \xleftrightarrow{\text{ $K''$ }} (x,y),  
\end{split}
\end{equation}
while applying the same operators alternately starting with $K''$,
generates the same sequence of points  but in reverse order.

Thus, it can be seen that the pairs of integers that appear in the sequence~\eqref{eqSeqab}
form a cycle that typically consists of six points, arranged geometrically as the nodes of 
a closed twisted path.
In fact, with the notation introduced in~\eqref{eqPKL},
\begin{equation*}\label{}
    \# \scP_K(a,b)\in \{ 1,3,6\}\,,
\end{equation*}
since the paths only have three points in case one of the points 
lies on one of the lines 
$y = -x+2$,  
$2y = x+1$,  
$y = 2x-1$, and exactly one, meaning all six points coincide at~$(1,1)$, the intersection of all three lines.
Except for these cases, when some points appear overlapped repeatedly, the union of all cycles forms a disjoint partition of $\ZZ^2$.
It can be seen in Figures~\ref{FigureCycles0M} and~\ref{FigureCycles-MM} that the cycles
build the enlarging figures that tend to cover the plane.

\subsection{Cycles representing squares and cubes}
Let us observe that the definition of the operators $K'$ and $K''$ causes only $6$ out of the 
$12$ integers that are components of the pairs in $\scP_K(a,b)$
to be distinct.
Looking only at the arithmetic properties of the represented integers 
in a cycle, we denote the set of their absolute values by
\begin{equation*}\label{}
    \cR_K^{{|\cdot|}}(x,y):= \big\{ |m| : m\in \{a,b\},\ (a,b) \in\cP_K(x,y)\big\}\,.
\end{equation*}

\begin{figure}[ht]
 \centering
 \mbox{
 \subfigure{
    \includegraphics[width=0.49\textwidth]{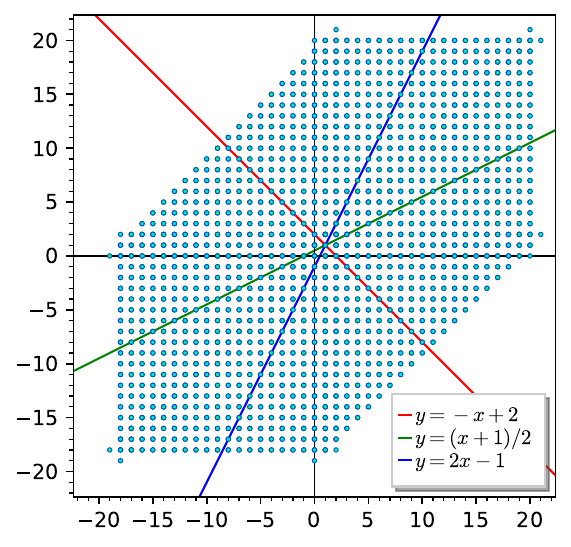}
 }
 \subfigure{
    \includegraphics[width=0.486\textwidth]{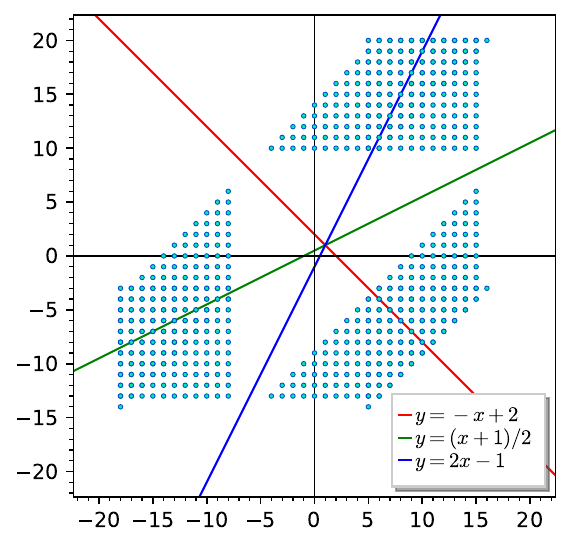}
 }
 }\hfill\mbox{}
\caption{
The union of all $K$-generated cycles that have at least one node in \mbox{$[0,20]\times [0,20]$} (image on the left)
and at least one node in $[5,15]\times [10,20]$ (image on the right).
}
 \label{FigureCycles}
 \end{figure}

\begin{figure}[ht]
 \centering
 \mbox{
 \subfigure{
    \includegraphics[width=0.49\textwidth]{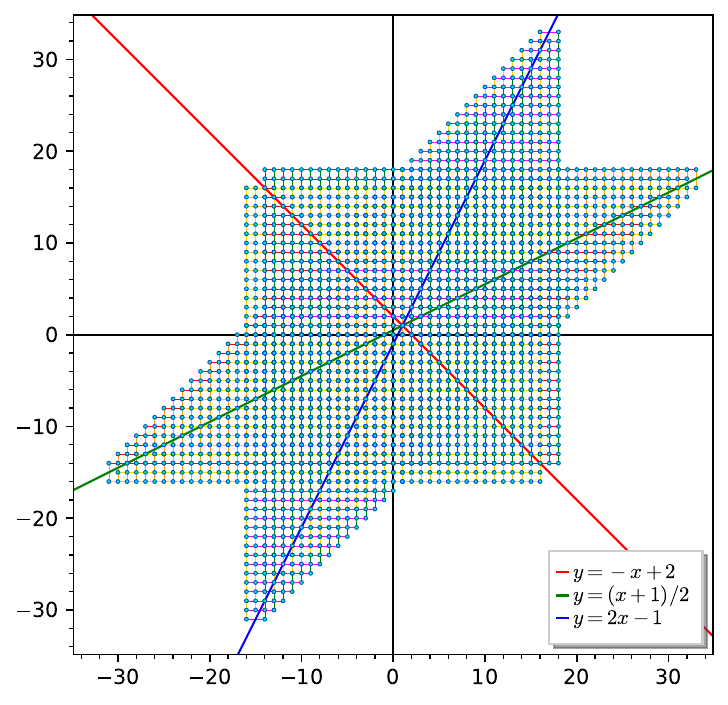}
 }
 \subfigure{
    \includegraphics[width=0.486\textwidth]{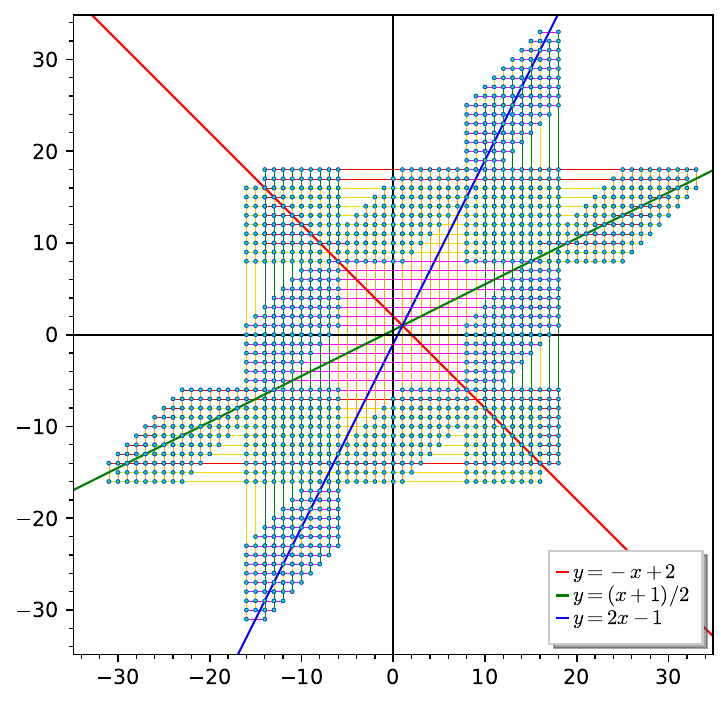}
 }
 }\hfill\mbox{}
\caption{
All nodes and paths that connect them in the $K$-generated cycles that have at least one node in $[-16,16]\times [-16,16]$ (image on the left)
and at least one node in $\big([-16,-8]\cup [8,16]\big)\times \big([-16,-8]\cup [8,16]\big)$ (image on the right).
}
 \label{FigurePaths}
 \end{figure}

The natural numbers in $\RmK(x,y)$ can have, both as a set and in relation to each other, 
various interesting arithmetic properties.
First, let us note that depending on $x$ and $y$, some elements may be equal, causing $\RmK(x,y)$ to have fewer than six elements.
These special cases do occur, specifically on the axes of symmetry of the enlarging sets of points 
shown in Figures~\ref{FigureCycles0M},
\ref{FigureCycles-MM},
\ref{FigureCycles},
\ref{FigurePaths}.
For example
\begin{align*}
    \RmK(1,1)&=\{1 \}; & \RmK(3,3)&=\{1, 3\}; & \RmK(3,5)&=\{1, 3, 5 \};\\
    \RmK(3,2)&=\{0, 1, 2, 3 \}; & \RmK(-2,2)&= \{0, 2, 3, 4, 5 \}; & \RmK(5,6)&=\{0, 2, 3, 4, 5, 6 \}.
\end{align*}

A special type of sets are those that contain powers of the same order, such as:
\begin{align*}
    \RmK(51,36)&=\{14, 4^2, 34, 6^2, 7^2, 51 \}; \\
  \RmK(64, 144) &= \{62, 8^2, 79, 9^2, 142, 12^2\}; \\    
    \RmK(512, 729) &= \{6^3, 218, 510, 8^3, 727, 9^3\};\\
  \RmK(830584,1092727) &= \{262142, 64^3, 830582, 94^3, 1092725, 103^3\}.
\end{align*}


In general, following the definition, 
for $x> y\ge 2$, including the  possible 
multiplicities, $\RmK(x,y)$ can be expressed as:
\begin{equation}\label{eqzexpy1}
   \cR_K^{{|\cdot|}}(x,y)
   = \big\{x,x-2,y,y-2,z,z-2\big\}, \ \text{ where $z=x-y+1$}.
\end{equation}
Then, since in $\RmK(x,y)$ three elements are 
close, with just $2$ less than the other three, 
no more than three squares or higher powers can appear in $\RmK(x,y)$.

Cases with exactly three squares in $\RmK(x,y)$ are many. Indeed, if we put
$x=A^2$, $y=B^2$ and $z-2=C^2$, according to~\eqref{eqzexpy1}, they need to satisfy
condition $C^2+2=A^2-B^2+1$, that is,
$C^2+1 = (A-B)(A+B)$.
Following a modulo $4$ analysis and taking $A-B=1$, we obtain a particular infinite sequence of
$6$-tuples  $\RmK(x,y)$  containing three squares each.
Their generators are parameterized by the pairs
\begin{equation}\label{eqGSquare1}
\big(x_1(t),y_1(t)\big) = \big((2t^2+1)^2, (2t^2)^2\big),\ \text{ for integers $t\ge 1$,} 
\end{equation}
and the third element of  $\RmK\big(x_1(t),y_1(t)\big)$, which is also a square, is $z_1(t)=(2t)^2$.

Likewise, we obtain another parameterization of a family of generators for the sets  $\RmK(x,y)$  
that contain three squares each, taking this time
$x=A^2$, $y=B^2$ and $z=C^2$.
In that case, the resulting analogous parameterization is
\begin{equation}\label{eqGSquare2}
  \big(x_2(t),y_2(t)\big) = \big((2t)^2, (2t^2-1)^2\big),
\end{equation}
with the third square in  $\RmK\big(x_2(t),y_2(t)\big)$  being $z_2(t)=(2t)^2$.

\begin{center}
\setlength\belowcaptionskip{12pt} 
\setlength\abovecaptionskip{12pt}
\setlength{\intextsep}{-2pt} 
\begin{table}[thb]
 \centering
 \captionsetup{width=.91\linewidth}
  \caption{
  The sets of integers represented by parameterization~\eqref{eqGSquare1}
  and~\eqref{eqGSquare2} for $2\le t\le 10$, where
  $ (x_1(t),y_1(t)) = \big((2t^2+1)^2, (2t^2)^2\big)$ and 
  $ (x_2(t),y_2(t)) = \big((2t)^2, (2t^2-1)^2\big)$.}
 \vspace{-5pt}
 \setlength{\tabcolsep}{3pt}
\renewcommand{\arraystretch}{1.21}
 \footnotesize
\begin{tabular}{@{}cccccc @{}}
\toprule 
$t$    & $ \big(x_1(t),y_1(t)\big)$ & $\RmK\big(x_1(t),y_1(t)\big)$ 
       && $ \big(x_2(t),y_2(t)\big)$ & $\RmK\big(x_2(t),y_2(t)\big)$
\\ \midrule
%
2 & $(81,64)$
& $\{4^2, 18, 62, 8^2, 79, 9^2\}$ &
 & $(64,49)$
& $\{14, 4^2, 47, 7^2, 62, 8^2\}$ \\
3 & $(361,324)$
& $\{6^2, 38, 322, 18^2, 359, 19^2\}$ &
 & $(324,289)$
& $\{34, 6^2, 287, 17^2, 322, 18^2\}$ \\
4 & $(1089,1024)$
& $\{8^2, 66, 1022, 32^2, 1087, 33^2\}$ &
 & $(1024,961)$
& $\{62, 8^2, 959, 31^2, 1022, 32^2\}$ \\
5 & $(2601,2500)$
& $\{10^2, 102, 2498, 50^2, 2599, 51^2\}$ &
 & $(2500,2401)$
& $\{98, 10^2, 2399, 49^2, 2498, 50^2\}$ \\
6 & $(5329,5184)$
& $\{12^2, 146, 5182, 72^2, 5327, 73^2\}$ &
 & $(5184,5041)$
& $\{142, 12^2, 5039, 71^2, 5182, 72^2\}$ \\
7 & $(9801,9604)$
& $\{14^2, 198, 9602, 98^2, 9799, 99^2\}$ &
 & $(9604,9409)$
& $\{194, 14^2, 9407, 97^2, 9602, 98^2\}$ \\
8 & $(16641,16384)$
& $\{16^2, 258, 16382, 128^2, 16639, 129^2\}$ &
 & $(16384,16129)$
& $\{254, 16^2, 16127, 127^2, 16382, 128^2\}$ \\
9 & $(26569,26244)$
& $\{18^2, 326, 26242, 162^2, 26567, 163^2\}$ &
 & $(26244,25921)$
& $\{322, 18^2, 25919, 161^2, 26242, 162^2\}$ \\
10 & $(40401,40000)$
& $\{20^2, 402, 39998, 200^2, 40399, 201^2\}$ &
 & $(40000,39601)$
& $\{398, 20^2, 39599, 199^2, 39998, 200^2\}$ \\
    \bottomrule
\end{tabular} \label{Table1}
\end{table}
\end{center}


For cubes, a family similar to parameterizations~\eqref{eqGSquare1} and~\eqref{eqGSquare2} is generated by
\begin{equation}\label{eqGCubes}
  \big(x_3(t),y_3(t)\big) = \big((9t^4+3t)^3, (9t^3+1)^3\big),
\end{equation}
and the third cube in $\RmK\big(x_3(t),y_3(t)\big)$  is $z_3(t)=\big(9t^4\big)^3$.
This is implied by Mahler's solution~\cite{Mah1936} for the 
three cubes Diophantine equation
\begin{equation}\label{eqThreeCubes}
  x^3+y^3+z^3 = 1.
\end{equation}
\begin{center}
\setlength\belowcaptionskip{12pt} 
\setlength\abovecaptionskip{12pt}
\setlength{\intextsep}{-2pt} 
\begin{table}[thb]
 \centering
 \captionsetup{width=.91\linewidth}
  \caption{
  The set of integers represented by parameterization~\eqref{eqGCubes} for $1\le t\le 10$, where
  $ (x_3(t),y_3(t)) =  \big((9t^4+3t)^3, (9t^3+1)^3\big)$.}
 \vspace{-5pt}
 \setlength{\tabcolsep}{-2pt}
\renewcommand{\arraystretch}{1.21}
 \footnotesize
\begin{tabular}{@{}ccc @{}}
\toprule 
$t$    & $ \big(x_3(t),y_3(t)\big)$ & $\RmK\big(x_3(t),y_3(t)\big)$
\\ \midrule
1 & $(1728, 1000)$
& $\{727, 9^3, 998, 10^3, 1726, 12^3\}$ \\
%
2 & $(3375000, 389017)$
& $\{389015, 73^3, 2985982, 144^3, 3374998, 150^3\}$ \\
%
3 & $(401947272, 14526784)$
& $\{14526782, 244^3, 387420487, 729^3, 401947270, 738^3\}$ \\
%
4 & $(12422690496, 192100033)$
& $\{192100031, 577^3, 12230590462, 2304^3, 12422690494, 2316^3\}$ \\
%
5 & $(179406144000, 1427628376)$
& $\{1427628374, 1126^3, 177978515623, 5625^3, 179406143998, 5640^3\}$ \\
%
6 & $(1594232306568, 7357983625)$
& $\{7357983623, 1945^3, 1586874322942, 11664^3, 1594232306566, 11682^3\}$ \\
%
7 & \, $(10119744747000, 29446377472)$\, 
& $\{29446377470, 3088^3, 10090298369527, 21609^3, 10119744746998, 21630^3\}$ \\
%
8 & $(36888^3, 4609^3)$
& $\{97908438527, 4609^3, 50096498540542, 36864^3, 50194406979070, 36888^3\}$ \\
%
9 & $(59076^3, 6562^3)$
& $\{282558696326, 6562^3, 205891132094647, 59049^3, 206173690790974, 59076^3\}$ \\
%
10 & $(90030^3, 9001^3)$
& $\{729243026999, 9001^3, 728999999999998, 90000^3, 729729243026998, 90030^3\}$ \\
    \bottomrule
\end{tabular} \label{Table2}
\end{table}
\end{center}

In addition, we mention that
Lehmer~\cite{Leh1956} obtained for~\eqref{eqThreeCubes} 
other types of recursively generated families of solutions. Many other types of solutions have been discovered by 
Payne and Vaserstein~\cite{PV1991}
(see~\cite[Table 2]{PV1991} for all solutions in a limited range and ~\cite[Theorem 1]{PV1991} for
infinitely many other families of solutions).
Payne and Vaserstein~\cite[Theorem 2]{PV1991}
also showed that there does not exist a finite set of
polynomial solutions that covers all integer solutions of~\eqref{eqThreeCubes}.

The first sets of integers generated by 
parameterizations~\eqref{eqGSquare1}, \eqref{eqGSquare2}
and~\eqref{eqGCubes}, which contain three squares and 
three cubes each, are given in Tables~\ref{Table1} and~\ref{Table2}.

\subsection{Cycles representing only primes}
Some other interesting sets of integers represented by cycles are:
\begin{align*}
	\RmK(13, 31)&=\{11, 13, 17, 19, 29, 31 \}; \\
  \RmK(103, 109) &= \{5, 7, 101, 103, 107, 109\}; \\    
    \RmK(601, 1033) &= \{431, 433, 599, 601, 1031, 1033\};\\
  \RmK(8431, 9859) &= \{1427, 1429, 8429, 8431, 9857, 9859\}.
\end{align*}
in which all six represented natural numbers are primes.

Since the elements in any $\RmK$ must verify condition~\eqref{eqzexpy1},
these sets can be expressed as the \textit{Goldbach-twin-prime} hex-tuples
\begin{equation}\label{eqsixprimes}
   \RmK(p,q)= \big\{p,q,p-q+1,p-2,q-2,p-q-1\big\}.
\end{equation}

By the works of Montgomery and Vaughn~\cite{MV1975}, Perelli and Pinz~\cite{PP1992}, and the more recent paper of Pintz~\cite{Pin2023}, we know that apart from
an exceptional set, the `Goldbach-prime' requirements on 
$z = p-q+1$ and $z-2=p-q-1$ are separately satisfied
infinitely often.
However, imposing the additional twin-prime condition makes it difficult 
to show the existence of infinitely many hex-prime triples~\eqref{eqsixprimes}.

An investigation in small ranges shows that there are actually plenty of 
sets $\RmK(x,y)$ that are composed of prime numbers only. 
Thus, let $\cP$ denote the set of prime numbers and, for any $T>1$, 
consider the set of pairs 
\begin{equation*}
    \cG(T):=\big\{\big(p,q) : 2\le p\le q\le T,\ \RmK(p,q)\subset\cP\big\}.
\end{equation*}
Then, if $T>110$,  
\begin{equation*}
\cG(T) = \big\{(7, 13);
(13, 19);
(19, 31);
(31, 43);
(43, 61);
(61, 73);
(73, 103);
(103, 109);\dots\big\}.
\end{equation*}
where we chose to show only the first pair from the lexicographically sorted $\cG(T)$ so that the first components of the pairs are different.
If $T=10^5$, then $\cG(T)$ has $\#\cG(10^5)=4120$ elements.

Naturally,  there are many pairs of primes $(p,q)\neq (p',q')$ 
for which $\RmK(p,q)=\RmK(p',q')$, but still,
$\cG(10^5)$ contains $2064$ pairs $(p,q)$ 
that generate distinct sets of primes $\RmK(p,q)$.

We remark that for $T\le 10^5$ the first components 
of the pairs in $\cG(T)$ are the largest of the primes in the 
twin prime pairs that are $\ge 7$, 
and this is likely to hold for larger $T$.

\subsection{\texorpdfstring{An $\ell^2$ average length of the steps of paths in $K$-generated cycles}{An L2 average length of the steps of paths in K-generated cycles}}

The size of steps in the path of a $K$-cycle generated by $(x,y)$ are, according to~\eqref{eqSeqab}, the terms of the following total length:
\begin{equation}\label{eqLengthPath}
   \length_K(x,y) := 
   |2x-y-1|+ |x+y-2| +|x-2y+1|+|2x-y-1|+|x+y-2|+|x-2y+1|.
   \end{equation}
In order to simplify the discussion on the signs of the six terms of
$\length_K(x,y)$, we introduce $S(x,y)$ defined to be the sum of their squares. 
This proves to be
\begin{equation*}
   S(x,y) = 
   12(x^2+y^2-xy-x-y+1).
\end{equation*}
Then, for integers $T\ge 1$, we define the square-steps average by
\begin{equation}\label{eqAp}
   A_{step}(T) := \frac{1}{\sqrt{6}(2T+1)}\left(\sum_{x=-T}^T\sum_{y=-T}^T
                S(x,y)  \right)^{1/2}.
\end{equation}
The double sum above can be calculated exactly. Indeed, we have:
\begin{equation*}
  \begin{split}
   \sum_{x=-T}^T\sum_{y=-T}^T S(x,y)
   &= \sum_{x=-T}^T \Big(8T^3 + 12(2T + 1)x^2 + 12T^2 - 12(2T + 1)x + 28T + 12\Big)\\
   &= 4\big(8T^4 + 16T^3 + 22T^2 + 14T + 3\big)\\
   &= 4(2T+1)^2\big(2T^2 + 2T + 3\big).
  \end{split}
\end{equation*}
On inserting this in~\eqref{eqAp}, we obtain
\begin{equation*}
  \begin{split}
   A_{step}(T) &= \sdfrac{2}{\sqrt{6}}
          \left( 2T^2 + 2T + 3  \right)^{1/2}
          = \sdfrac{2}{\sqrt{3}}T
          \big( 1 + O(1/T)  \big)^{1/2} 
          = \sdfrac{2}{\sqrt{3}}T + O(1).
  \end{split}
\end{equation*}

\subsection{\texorpdfstring{The average length of paths in $K$-generated cycles}{The average length of paths in K-generated cycles}}
Let $T\ge 1$, and define
\begin{equation*}
   A_{path}(T) := \frac{1}{(2T+1)^2}\sum_{x=-T}^T\sum_{y=-T}^T
                \length(x,y).
\end{equation*}

Note that both 
$A_{step}(T)$ and $A_{path}(T)$ are slightly biased
being weighted sums, since each path is counted in the sums as many times 
as the number of its nodes that lie in $[-T,T]^2$.
Thus, the multiplicities of the appearance of these paths in the sums are $1,2,3,4,6$.

Due to the symmetries, the steps in a cycle may have only two different lengths, so that
\begin{equation*}
   \length_K(x,y) = 4|2x-y-1|+ 2|x+y-2|.
\end{equation*}
Further, it follows that
\begin{equation}\label{eqAint}
  \begin{split}
      A_{path}(T) &= \frac{1}{(2T+1)^2}\int_{-T}^T\int_{-T}^T
        \big(4|2x-y-1|+ 2|x+y-2|\big)\,\dd x\dd y +O(1)\\
      &= \frac{1}{4T^2}\int_{-T}^T\int_{-T}^T
        \big(4|2x-y|+ 2|x+y|\big)\,\dd x\dd y +O(1),
  \end{split}
\end{equation}
since on each unit square $u(x,y)\subset [-T,T]^2$, the size of
$\length(x,y)$ differs by at most a fixed independent, absolute constant from the double integral of
$\length(x,y)$ on $u(x,y)$.
Then, taking into account the signs of the terms, \eqref{eqAint} becomes
\begin{equation*}
  \begin{split}
      A_{path}(T) 
      =& \frac{1}{2T^2}\int_{-T/2}^{T/2}\left(\int_{-T}^{2x}
        4(2x-y)\,\dd y\right)\dd x
        + \frac{1}{2T^2}\int_{T/2}^{T}\left(\int_{-T}^{T}
        4(2x-y)\,\dd y\right)\dd x\\
        &+ \frac{1}{2T^2}\int_{-T}^{T}\left(\int_{-x}^{T}
        2(x+y)\,\dd y\right)\dd x  +O(1).
  \end{split}
\end{equation*}
Evaluating the integrals, we find
\begin{equation*}
  \begin{split}
      A_{path}(T) 
        =& \frac{17}{3}T  +O(1).
  \end{split}
\end{equation*}

One should remark the difference between the average
length of the path given as the six times the $\ell^2$ average of the steps
by \mbox{$6A_{steps}=\sdfrac{12}{\sqrt{3}}T + O(1)$} and the shorter average estimate 
$A_{path}(T)=\frac{17}{3}T  +O(1)$,
fact influenced by the varying lengths of the steps in the same path.

\section{Proof of Theorem~\ref{TheoremMain}}\label{SectionProofThMain}
\subsection{One equivalence class}\label{subsectionTM1}
Let the pair $(a,b)\in\ZZ^2$ be fixed and suppose there exists $m\in\ZZ$
such that $m\le \min(x,y)$ for any $\scP(a,b)$. 
Suppose $m$ is the largest of these lower bounds.
Then, $m$ is a  component of a pair in $\scP(a,b)$,
which by the symmetry we may assume that it is the first, 
that is, there exists $u\in\ZZ$ such that $(m,u)\in\scP(a,b)$.

Then $L''(m,u)=(m,-u+2m+1)\in\scP_L(a,b)$, which implies
$m\le 2m-u+1$. Therefore $u\le m+1$. 
There are only two possibilities: either $u=m$ or $u=m+1$.

If $u=m$, then $(m,m)\in\scP_L(a,b)$, which implies that 
$\scP(a,b) = \scP(0,0) +m$.

If $u=m+1$, then we find that the pair $(m,m+1)$ is equivalent
with the pair $(m,m)$, because $L''(m,m+1)=(m,-(m+1)+2m+1)=(m,m)$.
Therefore, in this case, $(m,m)$ also belongs to $\scP(a,b)$.

In conclusion, any set of pairs $\scP(a,b)$ is equivalent with
$\scP(0,0)$ because they are the same modulo a translation by $m$
\begin{equation*}
    \scP(a,b) = \scP(0,0) +m,
\end{equation*}
where $m=\min\big\{\scR_L(a,b)\big\}$.
One can follow this relation over the parabolas represented in Figure~\ref{Figure17parabolas}.

\begin{figure}[ht]
 \centering
    \includegraphics[width=0.69\textwidth]
    {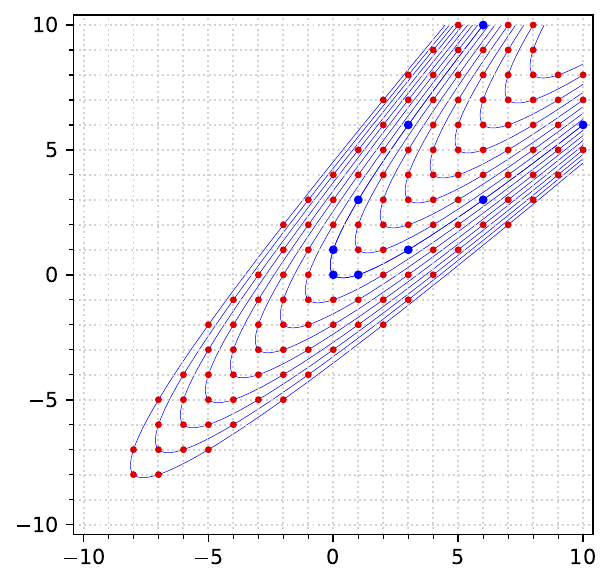}
\caption{
The parabolas of equations $x+y+2m = (x-y)^2$ for $m\in\{-8,-7,\dots,7,8\}$
and the highlighted lattice points
from $\scP_L(m,m)$ they pass through.
}
 \label{Figure17parabolas}
 \end{figure}

\subsection{The existence of \texorpdfstring{$\min\big\{\scR_L(a,b)\big\}$}{min(R(a,b))}}\label{subsectionTM2}

Let $a,b\in\ZZ$. 
First, we show that there is a point in $\scP_L(a,b)$ 
that has both coordinates equal.
Then, we will observe that 
starting from that point and iteratively applying the operators 
in $\{L',L''\}$ alternatively, 
we will see that none of the points obtained in this way will 
have coordinates smaller than those of its predecessor.

Let $(x,y)\in\scP_L(a,b)$ and let $s=\abs{x-y}$ be the distance
between the coordinates. By the definition, we gave
\begin{equation}\label{eqLL}
    L'(x,y)=(-x+2y+1,y)\ \ \text{ and }\ \ 
    L''(x,y) = (x,2x-y+1).
\end{equation}
Then, the distances between the coordinates of 
the points obtained are
\begin{equation}\label{eqabs}
    \abs{-x+2y+1-y} = \abs{(x-y)-1}
    \ \ \text{ and }\ \
    \abs{2x-y+1-x} = \abs{(x-y)+1}.      
\end{equation}
This shows that, unless $x=y$, the distances between 
the coordinates of points $L'(x,y)$ and $L''(x,y)$ 
are equal to $s-1$ and $s+1$ respectively. 
By repeatedly applying the operators $L'$ and $L''$ 
on those points with coordinates that gradually approach each 
other, after a total of $s$ steps, we will ultimately reach a point in $\scP_L(a,b)$ that has both coordinates equal.

We should also note one more thing in regard to~\eqref{eqLL} and~\eqref{eqabs}.
By checking separately the cases $x>y$ and $x<y$,
we further find that when 
the components of the points move away from each other, 
their size does not decrease.

As a consequence, we have shown that there exists $m\in\ZZ$
such that \mbox{$(m,m)\in\scP_L(a,b)$} and
$m=\min\big\{\scR_L(a,b)\big\}$.

\subsection{The ladder on \texorpdfstring{$\scP(0,0)$}{P(0,0)}}\label{subsectionTM3}
Note that since we know from Lemma~\ref{LemmaInvolutionsGeneral} 
that $L'$ and~$L''$ are involutions, in order to move out of the
$2$-length cycles, we must apply $L'$ and $L''$ alternatively.
Then, starting with $(0,0)$, 
and following the steps in  Lemma~\ref{LemmaLL},
we obtain the following sequences:
\begin{equation*}
    \begin{split}
    (0,0) &\xrightarrow{\text{ $L''$ }} (0,1)
      \xrightarrow{\text{ $L'$ }} (3,1)
      \xrightarrow{\text{ $L''$ }} (3,6)
       \xrightarrow{\text{ $L'$ }} (10,6)
       \xrightarrow{\text{ $L''$ }} (10,15)
       \xrightarrow{\text{ $L'$ }}\cdots,\\
    (0,0) &\xrightarrow{\text{ $L'$ }} (1,0)
      \xrightarrow{\text{ $L''$ }} (1,3)
      \xrightarrow{\text{ $L'$ }} (6,3)
       \xrightarrow{\text{ $L''$ }} (6,10)
       \xrightarrow{\text{ $L'$ }} (15,10)
       \xrightarrow{\text{ $L''$ }}\cdots.     
  \end{split}
\end{equation*}
Then, by induction for $k\ge 0$, using the formulas
\begin{equation*}
    \begin{split}
    L'(T_k,T_{k+1}) &=(-T_k+2T_{k+1}+1,T_{k+1})=(T_{k+2},T_{k+1}), \\
    L''(T_{k+1},T_k) &=(T_{k+1},-T_k+2T_{k+1}+1,)=(T_{k+1},T_{k+2}),
  \end{split}
\end{equation*}
where $T_k=k(k+1)/2$, we obtain the general expression
for the points in $\scP_L(0,0)$, as shown in Figure~\ref{FigureLadders}.

In particular, if we let $(x,y)=(T_k,T_{k+1})$ or
$(x,y)=(T_{k+1},T_k)$, points that are symmetric with respect to
$x=y$, we find that $x+y = (x-y)^2$, 
which is the equation of the parabola on which lie all points of $\scP_L(0,0)$.

\subsection{The vertex of the parabola passing through  a given point}\label{subsectionTM4}

Let $(a,b)\in\ZZ$ be fixed.
On combining the results from Sections~\ref{subsectionTM1}
and~\ref{subsectionTM3} it follows that there exists
an integer $m$ such that $(a,b)\in\scP_L(a,b)$ and
$(a,b)=(T_k+m,T_{k+1}+m)$, for some $k\ge 0$.

On subtracting the coordinates of the point, it follows that
$b-a = k+1$, so that 
\begin{equation}\label{eqm1}
    k=b-a-1.
\end{equation}
Therefore, m is the solution of the equation
$a = T_k+m$, that is,
\begin{equation}\label{eqm2}
    m = a - T_k.
\end{equation}
Inserting the value of $k$ from~\eqref{eqm1} in~\eqref{eqm2},
we find
\begin{equation*}
    m = a - \frac{(b-1)(b-a-1)}{2} 
     =\frac{(b+a)-(b-a)^2}{2}.
\end{equation*}
These concludes the proof of Theorem~\ref{TheoremMain}.

\section{The densities of represented numbers in \texorpdfstring{$\scR_L(a,b) \pmod p$}{mod p} -- Proof of Theorem~\ref{TheoremLmodp}}\label{SectionLDensities}

Since we know from Theorem~\ref{TheoremMain}
that $\scP_L(a,b)=\scP_L(0,0)+m$, where 
\begin{equation*}
   m=\big(a+b-(a-b)^2\big)/2    
\end{equation*}
and
$(m,m)$
is the vertex of the parabola on which lie the lattice points 
in  $\scP_L(a,b)$, it is enough to find the densities
of the represented numbers in the residue classes $\pmod p$
only for the numbers in $\scR_L(0,0)$.

For a first approach to the subject of distribution of represented integers
on $\scP_L(0,0)$
in residue classes, we have included in 
Figure~\ref{FigureParabola-BarColor-mod17}
a graphical presentation of the densities in the specific case $p=17$.
Even more revealing for the densities are the four histograms of the
frequencies of the residues classes 
of the represented integers on the parabola $\cP(0,0)$
that are shown in Figure~\ref{FigureHistograms13101}.

\begin{figure}[ht]
 \centering
 \subfigure{
    \includegraphics[width=0.69\textwidth]{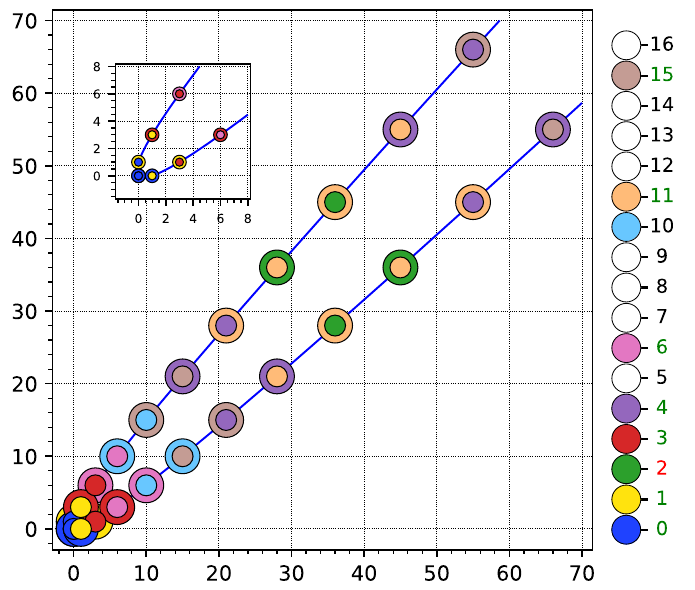}
 \label{FigureRC}
 }
\caption{
Starting from $(0,0)$, 
the points obtained by iterating $L'$ and~$L''$ in any order
are shown covered with bicolored discs.
The colors indicate belonging to the classes of representatives modulo $17$, 
so that the inner color corresponds to the first coordinate and 
the outer color corresponds to the second one.
The only classes with non-zero density are $0, 1, 2, 3, 4, 6, 10, 11, 15$ and $2$ 
appears with half the density compared to each of the other non-zero ones.
}
 \label{FigureParabola-BarColor-mod17}
 \end{figure}

In order to estimate the number $N_L(l,p;T)$ defined by~\eqref{eqN}, we use the parameterization
from the second part of Theorem~\ref{TheoremMain}, which 
shows that 
$\scR_L(0,0) = \{\frac{k(k+1)}{2} : k\ge 0\}$.

\smallskip
If $p=2$, as a result of the periodic repetition of
integers divisible by $4$, it follows that
\begin{equation*}
    \bigabs{N_L(0,2;T)-N_L(1,2;T)} \le 1,
\end{equation*}
for any $T>0$,  which implies that the limit density defined by~\eqref{eqLimitDensity} is
$\dens_L(0,2)=\dens_L(1,2)=\frac 12$.

\smallskip
If $p=3$, one finds that there are no integers $k$
for which $k(k+1)/2\equiv 2 \pmod 3$, while for the other residue classes, $0$ is about twice as frequent as $1$, and we have
\begin{equation*}
    \bigabs{N_L(0,3;T)-2N_L(1,3;T)} \le 1,
\end{equation*}
for any $T>0$.
Then $\dens_L(0,3)=\frac 23, \dens_L(1,3)=\frac 13$.
and $\dens_L(2,3)=0$.

A similar phenomenon occurs for $p=5$.
Indeed, there are no integers $k$
for which $k(k+1)/2\equiv 2 \text{ or } 4 \pmod 3$,
and out of the remaining three classes, 
one is represented half as much as the other two. The exact estimates are:
\begin{equation*}
    \bigabs{N_L(0,5;T)-N_L(1,5;T)} \le 1,
    \ \ \text{ and }\ \ 
    \bigabs{N_L(0,5;T)-2N_L(3,5;T)} \le 1,
\end{equation*}
for any $T>0$.
Then the densities are: $\dens_L(0,5)=\dd_L(1,5)=\frac 25$,
$\dens_L(3,5)=\frac 15$,
and $\dens_L(2,5)=\dens_L(4,5)=0$.

For any prime $p>2$, after completing the square, the congruence 
 $k(k+1)/2\equiv l \pmod p$
can be written in the equivalent form
\begin{equation}\label{eqpatrat}
    (k+2^{-1})^2\equiv 2l+2^{-2}\pmod p,
\end{equation}
where the inverses are taken modulo $p$.

With $x=k+2^{-1}$, congruence \eqref{eqpatrat} is an equation modulo $p$ 
in $x$, which has zero, one or two solutions, depending on
the constant term $2l+2^{-2}$.
The distinctions are as follows: there are zero solutions when
$2l+2^{-1}$ is a quadratic nonresidue $\pmod p$;
one solution when
$l=-2^{-3}$ (the representative of $-2^{-3}\pmod p$ in $[1,p-1]$);
and two solutions when $2l+2^{-2}$ is a non-zero quadratic residue
$\pmod p$.
Further, since $x=k+2^{-1}$ varies linearly with $k$ and the solutions
are distributed periodically in the $p$-length sub-intervals of  
of $[0,T]$, for $T\ge p$, it follows that
\begin{equation}\label{eqNLpl}
    \begin{cases}
  N_L(l,p;T) = 0, & \text{ if } \legendre[p]{2l+2^{-2}}=-1; \\[2mm]
  \bigabs{2N_L(l,p;T)-N_L(l_1,p;T)} \le 1, & 
  \text{ if } l=-2^{-3} \text{ and } 
  \legendre[p]{2l_1+2^{-2}}=1;\\[2mm]
  \bigabs{N_L(l_1,p;T)-N_L(l_2,p;T)} \le 1, & 
  \text{ if } \legendre[p]{2l_1+2^{-2}}=1  \text{ and } 
  \legendre[p]{2l_2+2^{-2}}=1.
    \end{cases}
\end{equation}
These estimates imply the limit values of the densities in 
relation~\eqref{eqdensLpl}, which concludes the proof of Theorem~\ref{TheoremLmodp}.

\begin{figure}[ht]
 \centering
 \subfigure{
    \includegraphics[width=0.24\textwidth]{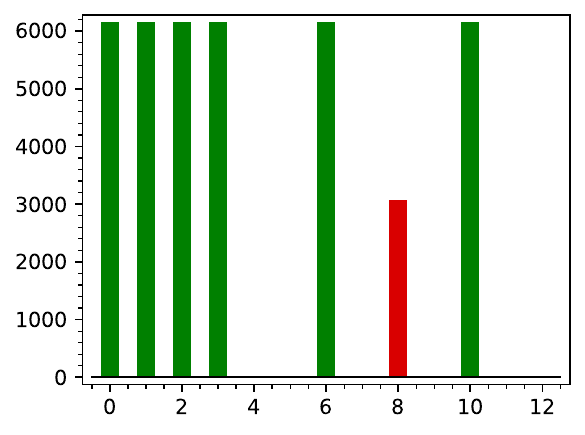}\
    \includegraphics[width=0.24\textwidth]{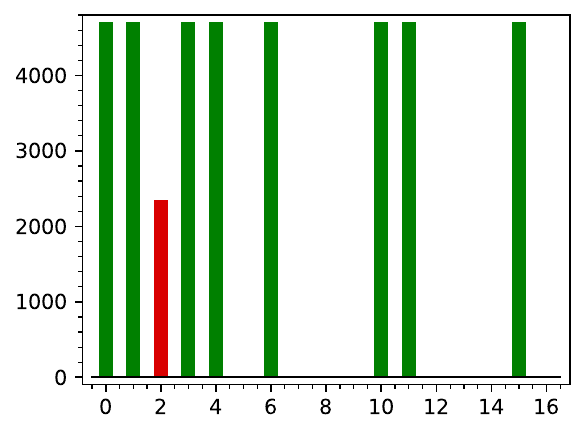} \   
        \includegraphics[width=0.24\textwidth]{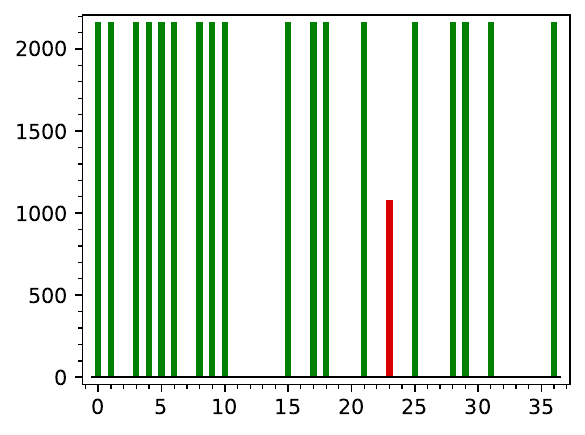}\
    \includegraphics[width=0.24\textwidth]{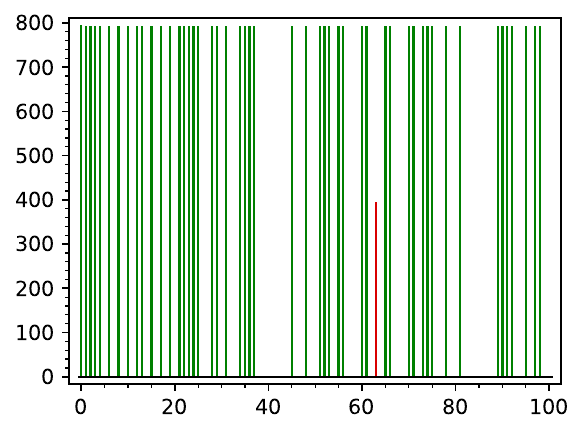} \   
 }
\caption{
The densities modulo $p$ of the represented integers 
on the parabola 
$\cP_L{(0,0)}$, which is started from $(0,0)$ 
and contains the points obtained by applying repeatedly the operators $L'$ and $L''$. The represented integers are taken modulo $p$ for 
$p=13, 17, 37,101$.
The histograms show the frequency of the residue classes of the coordinates of the vertex 
$(0,0)$ of the parabola and the nearest $10^4$ points to it.
}
 \label{FigureHistograms13101}
 \end{figure}
In Figure~\ref{FigureHistograms13101}, the frequencies 
$N_L(l,p;T)$ of the represented integers in $\scR_L(0,0)$
taken \mbox{modulo~$p$}
are shown for four values of $p$ and $0\le l<p$,
and $T=10^4$. 
The exceptional values, corresponding to the median densities, 
are highlighted in red. One checks that they are:
$l=8$ if $p=13$;
$l=2$ if $p=17$;
$l=23$ if $p=37$; and
$l=63$ if $p=101$,
and notes that $l\equiv -2^{-3}\pmod p$ in each of the four cases.

\section{The parabolic-taxicab distance}\label{SectionDistance}
Besides the operators $K$ and $L$ defined by~\eqref{eqK} and~\eqref{eqL},
we consider a third operator that enables 
movement horizontally  and vertically towards the closest neighbors:
\begin{align}\label{eqM}
    M'(x,y)& = (x+1,y),& M''(x,y)& = (x-1,y),\\
    M'''(x,y)& = (x,y+1),& M^{iv}(x,y)& = (x,y-1).
\end{align}

Denote $\cM:=\{M',M'',M''',M^{iv}\}$.
Then, the operators in $\cM$ can be used to define 
$\dM(P,Q)$,
the \textit{taxicab distance} between two points $P$ and $Q$, 
as the minimum length of the path that connects points $P$ and $Q$ 
by moving on the square grid that connects the lattice points with integer coordinates, that is,
\begin{equation}\label{eqDefinitionDM}
    \dM(P,Q) := \min\big\{ k\ge 0 : 
    \varphi^{[k]}(P)=Q,\,
    \varphi^{[k]} = \varphi_1 \circ\dots\circ\varphi_k,\,
    \varphi_j\in\cM,\, 1\le j\le k
    \big\}.
\end{equation}
Similarly, we define the \textit{parabolic-taxicab} distance distance $\dP(P,Q)$, which counts as units the successive application of operators from $\cL\cup\cM$ 
to generate a path of minimum length that connects $P$ and~$Q$:
\begin{equation}\label{eqDefinitionDP}
    \dP(P,Q) := \min\big\{ k\ge 0 : 
    \varphi^{[k]}(P)=Q,\, 
    \varphi^{[k]} = \varphi_1 \circ\dots\circ\varphi_k,\,
    \varphi_j\in\cL\cup\cM,\, 1\le j\le k
    \big\}.
\end{equation}
For example, the shortest path described in Figure~\ref{FigureTwoPaths} is
\begin{equation*}
  \begin{split}
	(5, 2)&\xrightarrow{\text{ $L''$ }} (5, 9)
    \xrightarrow{\text{ $L'$ }} 	(14, 9)
	\xrightarrow{\text{ $M^{iv}$ }} (14, 8)
	\xrightarrow{\text{ $M^{'}$ }} (15, 8)
	\xrightarrow{\text{ $L''$ }} (15, 23),
  \end{split}
\end{equation*}
so that $\dP\big((5, 2),(15, 23)\big)=5$.

In the next proposition we compare the parabolic-taxicab distance $\dP$
with the taxicab distance~$\dM$ and the Euclidean distance $\dE$.
\begin{proposition}\label{PropositionDP}
  The function defined by~\eqref{eqDefinitionDP} has the
  following properties.
  \begin{enumerate}
      \item $\dP$ is a distance;
      \item $\dE(P,Q)\le \dM(P,Q)$;
      \item $\dP(P,Q)\le \dM(P,Q)$,
  \end{enumerate}  
  for any lattice points $P,Q\in\ZZ^2$.
\end{proposition}
\begin{proof}
(1) The non-negativity and the identity conditions are both part of the definition
since, by convention, $\varphi^{[0]}=Id$. 

The symmetry follows since the inverses 
of $M', M'',M''',M^{iv}$ are in $\cM$
as well (since $(M')^{-1} = M''$ and
 $(M''')^{-1} = M^{iv}$)
and both $L'$ and $L''$ are involutions, 
according to Lemma~\ref{LemmaInvolutionsGeneral}.

The triangle inequality axiom is satisfied 
because any successive operators from $\cM\cup\cL$ 
applied to $P$ generate the nodes of a path 
that starts at $P$ and, in this specific case, ends at~$Q$, 
and, as defined, 
$\dP(P,Q)$ is given by the smallest number of steps required 
to go from $P$ to $Q$, while in $\dP(P,R)+\dP(R,Q)$ 
two minimum number of steps are added, which count also for paths from $P$ to $Q$,
but these are particular ones, as they additionally 
pass through~$R$, for any lattice points $P, Q, R$.

\smallskip
(2) This is the triangle inequality applied in a right triangle.

\smallskip
(3) The inequality follows because, while both sides calculate the minimum length 
of paths from $P$ to $Q$, the set of paths considered in calculating 
the minimum on the left side is larger, including the one used 
to calculate the minimum on the right side.
\end{proof}

Let us examine a few other examples that compare 
the distances between two points $P$ and $Q$, measured using $\dE, \dM$, and $\dP$.
Because all operators in $\cL\cup\cM$ are invariant under translations, we can choose $P=O=(0,0)$.

If $Q=(0,1)$, all distances are equal:
$\dE(O,Q)=\dM(O,Q)=\dP(O,Q)=1$.

If $Q=(1,1)$, we have:
$\dE(O,Q)=\sqrt{2}<\dM(O,Q)=\dP(O,Q)=2$,
because $O$ and $Q$ are neither adjacent on the lattice network
nor belong to the same parabola of those that partition $\ZZ^2$ as shown in Figure~\ref{Figure17parabolas}.

If $Q$ is farther away from the origin, than the parabolic-taxicab distance $\dP(O,Q)$ is smaller than $\dE(O,Q)$, 
because there are ladder steps on the parabolas that provides shorter paths. 
For example, if $Q=(3,6)$, 
then $\dE(P,Q)=\sqrt{6^2+3^2}=3\sqrt{5}\approx 6.7$
and \mbox{$\dM(O,Q)=3+6=9$}.
But $\dP(O,Q)=3$, because
\begin{equation*}
  \begin{split}
    (0,0) &\xrightarrow{\text{ $L''$ }} (0,1)
      \xrightarrow{\text{ $L'$ }} (3,1)
      \xrightarrow{\text{ $L''$ }} (3,6),
  \end{split}
\end{equation*}
so that $\big(L''\circ L'\circ L''\big)(0,0) = (3,6)$,
and one can check as well, case by case, that 
$(3,6)$ cannot be obtained from $(0,0)$
by applying fewer than $3$ operators from $\cL\cup\cM$.

\subsection{A way to farther away points}
In order to get a better understanding of the subject, let us evaluate the distance 
from the origin $O=(0,0)$ to a point $Y:=(2023,2024)$ located close but above the first diagonal.
The lattice points on the upper branch of the parabola with the minimum at $O$ have coordinates $(T_k,T_{k+1})$,
for any integer $k\ge 0$,
and the closest ones to~$Y$ are $Z:=(T_{62},T_{63})=(1953,2016)$ and $X:=(T_{63},T_{64})=(2016,2080)$.
Then, climbing the ladder in Figure~\ref{FigureLadders}, by following each pair of two steps up-right,
requires $2\cdot 63+1=127$ steps to reach from $O$ to $X$.
Next, taking the taxicab-lift down for $2080-2024=56$ steps, followed by 
$2023-2016=7$ steps to the right, we arrive at the destination $Y$.
All these intermediary steps, $\dP(O,X)\le 127$, $\dM\big(X,(2016,2024)\big)=56$
and \mbox{$\dM\big((2016,2024),Y\big)=7$},
add up to a total that shows that $\dP(O,Y)\le 127+56+7=190$.

Slightly different, someone could try a seemingly faster way 
by only going up to $Z$ instead of going all the way to $X$. 
But, then taking the taxicab to the right, we need to pay for 
\mbox{$2023-1953=70$} steps 
to reach $(2023,2016)$, followed by the lift-up for another \mbox{$2024-2016=8$} steps to reach the final destination $Y$.
The total then shows us that $\dP(O,Y)\le (2\cdot 62+1)+70+8=203$, which is more expensive compared to the previous choice.

However, it turns out that not even the first faster movement on the parabola 
to a point above followed by the taxicab-lift descent is the most efficient.
Indeed, there is an even less than three times shorter path from $O$ to $Y$
jumping and climbing on the nearby parabolas, which makes the parabolic-taxicab 
distance equal to just $\dP(O,Y)=83$.

\subsection{The \texorpdfstring{$\dP$}{dP}-ball}
For any lattice point $C=(a,b)$ and any integer $r\ge 0$
the \textit{parabolic-taxicab closed ball} is the set of lattice points that are at a distance at most~$r$ from $C$, that is,
\begin{equation*}
    \scBpc(C,r) :=\big\{
    X\in\ZZ^2 : \dP(C,X) \le r
    \big\}.
\end{equation*}
Then, the cardinality of $\scBpc(C,r)$ is by definition 
the \textit{measure} of the ball, which we denote by
$\mu\big(\scBpc(C,r)\big) = 
    \#\scBpc(C,r)$.

In Figure~\ref{FigureBalls}, are shown two balls of radius 
$10$ and $17$ both centered at the origin.
Although, one can identify some similarities with the shape 
of the \textit{swallow} (the set of points put together 
by the support of the distribution function
of the neighbor denominators of 
the Farey fractions~\cite{ABCZ2001,CVZ2010}),
the `wings' of the balls almost identify with the `tail', 
as if this were due to the superior speed on 
the parabolic trajectories.
On the other hand, balls for different radii 
have related shapes,
but they do not scale as in the Euclidean geometry.
Thus, in the case of geometry given by the distance~$\dP$, 
the analogue of $\pi$ is not an absolute constant, 
but it turns out to be dependent on $r$.

The step-by-step iterative calculation of the area of $\scBpc((O,r)$ for 
$r=0,1,2,3,\dots$ yields the sequence:
\begin{equation}\label{eqAB}
    1, 5, 15, 37, 75, 135, 221, 339, 493, 689, 931, 1225, 1575, 1987, 2465, 3015, 3641, 4349,\dots
\end{equation}
The gaps between neighbors elements of this sequence are:
\begin{equation*}
   4, 10, 22, 38, 60, 86, 118, 154, 196, 242, 294, 350, 412, 478, 550, 626, 708,\dots 
\end{equation*}
Next, the sequence of the gaps of the gaps of sequence~\eqref{eqAB} is:
\begin{equation*}
   6, 12, 16, 22, 26, 32, 36, 42, 46, 52, 56, 62, 66, 72, 76, 82,\dots 
\end{equation*}
Finally, the gaps of this last sequence alternate: 
$6, 4, 6, 4, 6, 4\dots$
This analysis then leads to the closed form formula~\eqref{eqAreaBall} in Conjecture~\ref{ConjectureBallArea}.

\begin{figure}[ht]
 \centering
  \includegraphics[width=0.24\textwidth]{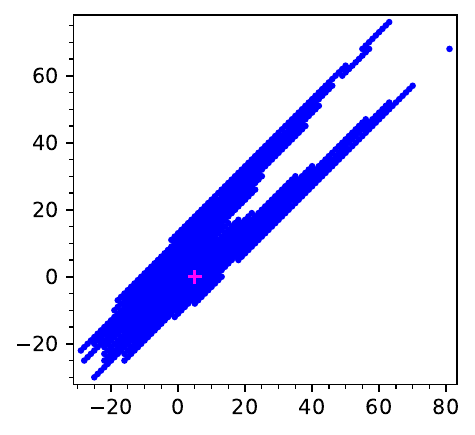}
  \includegraphics[width=0.24\textwidth]{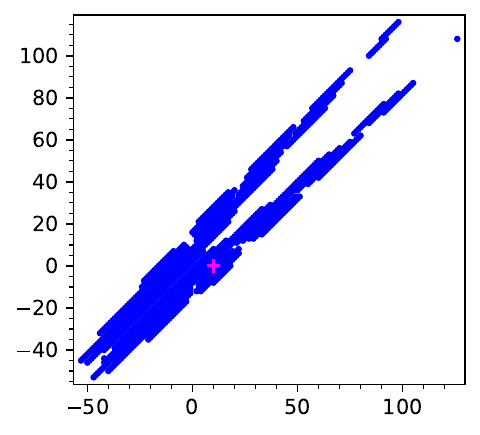}
  \includegraphics[width=0.24\textwidth]{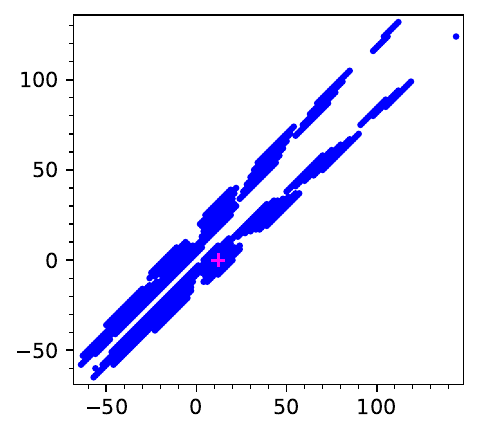}
  \includegraphics[width=0.24\textwidth]{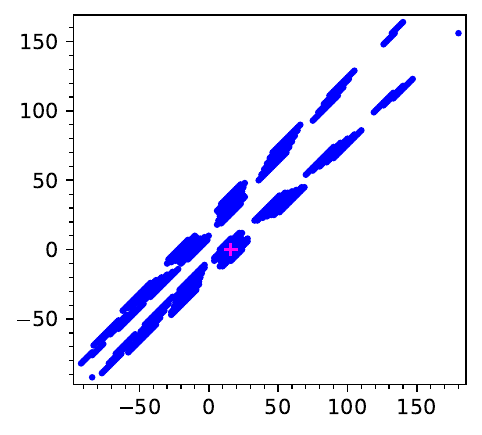}
\caption{
Parabolic-taxicab balls with centers off the principal diagonal. 
From left to right, their areas are:
$\mu\big(\scBpc((5,0),8)\big) = 1412$;
$\mu\big(\scBpc((10,0),8)\big) = 2250$;
$\mu\big(\scBpc((12,0),8)\big) = 2339$;
$\mu\big(\scBpc((16,0),8)\big) = 2361$.
}
 \label{FigureOffDiagonalBalls}
 \end{figure}

Balls that are not centered on the first diagonal are not symmetrical, 
and for small radii they are made up of 
clusters of points with a discontinuous appearance
(see Figure~\ref{FigureOffDiagonalBalls}).
As the radius increases, these clusters tend to merge into 
each other while new clusters form at the expanding ends 
in a direction parallel with the first diagonal.

\subsection*{Acknowledgements}
\bigskip
The authors acknowledge the essential role 
of their children Bianca, Ștefan and Dan,
dating back to the summer of 1997, 
for the intuitive description of the
parabolic-taxicab distance using the elevator
in the Wayside School.
Also, they thank Evghenia Obcinikova for fruitful discussions and for the link up to the bamboozle.


\begin{thebibliography}{10}

\bibitem{ABCZ2001}
Volker Augustin, Florin~P. Boca, Cristian Cobeli, and Alexandru Zaharescu.
\newblock The {{\(h\)}}-spacing distribution between {Farey} points.
\newblock {\em Math. Proc. Camb. Philos. Soc.}, 131(1):23--38, 2001.
\newblock \url{https://doi.org/10.1017/S0305004101005187}.
\newblock \href {https://doi.org/10.1017/S0305004101005187} {\path{doi:10.1017/S0305004101005187}}.

\bibitem{BCZ2023b}
Raghavendra~N. Bhat, Cristian Cobeli, and Alexandru Zaharescu.
\newblock A lozenge triangulation of the plane with integers.
\newblock {\em arXiv:2403.10500}, pages 1--21, 2024.
\newblock \url{https://arxiv.org/abs/2403.10500}.

\bibitem{BCZ2023}
Raghavendra~N. Bhat, Cristian Cobeli, and Alexandru Zaharescu.
\newblock On quasi-periodicity in {Proth-Gilbreath} triangles.
\newblock {\em Bull. Math. Soc. Sci. Math. Roum., Nouv. S{\'e}r.}, 67(1):3--21, 2034.
\newblock \url{https://ssmr.ro/bulletin/volumes/67-1/node2.html}.

\bibitem{CZZ2013}
Mihai Caragiu, Alexandru Zaharescu, and Mohammad Zaki.
\newblock An analogue of the {Proth}-{Gilbreath} conjecture.
\newblock {\em Far East J. Math. Sci. (FJMS)}, 81(1):1--12, 2013.
\newblock \url{http://www.pphmj.com/abstract/7973.htm}.

\bibitem{Ciucu2005}
Mihai Ciucu.
\newblock {\em A random tiling model for two dimensional electrostatics}, volume 839 of {\em Mem. Am. Math. Soc.}
\newblock Providence, RI: American Mathematical Society (AMS), 2005.
\newblock \url{https://doi.org/10.1090/memo/0839}.
\newblock \href {https://doi.org/10.1090/memo/0839} {\path{doi:10.1090/memo/0839}}.

\bibitem{Ciucu2009}
Mihai Ciucu.
\newblock {\em The scaling limit of the correlation of holes on the triangular lattice with periodic boundary conditions}, volume 935 of {\em Mem. Am. Math. Soc.}
\newblock Providence, RI: American Mathematical Society (AMS), 2009.
\newblock \url{https://doi.org/10.1090/memo/0935}.
\newblock \href {https://doi.org/10.1090/memo/0935} {\path{doi:10.1090/memo/0935}}.

\bibitem{CF2023}
Mihai Ciucu and Ilse Fischer.
\newblock Lozenge tilings of hexagons with removed core and satellites.
\newblock {\em Ann. Inst. Henri {Poincaré} Comb. Phys. Interact.}, 10(3):407--501, 2023.
\newblock \url{https://doi.org/10.4171/AIHPD/131}.
\newblock \href {https://doi.org/10.4171/AIHPD/131} {\path{doi:10.4171/AIHPD/131}}.

\bibitem{CL2019}
Mihai Ciucu and Tri Lai.
\newblock Lozenge tilings of doubly-intruded hexagons.
\newblock {\em J. Comb. Theory, Ser. A}, 167:294--339, 2019.
\newblock \url{https://doi.org/10.1016/j.jcta.2019.05.004}.
\newblock \href {https://doi.org/10.1016/j.jcta.2019.05.004} {\path{doi:10.1016/j.jcta.2019.05.004}}.

\bibitem{CCZ2000}
C.~I. Cobeli, M.~Cr{\^a}{\c{s}}maru, and A.~Zaharescu.
\newblock A cellular automaton on a torus.
\newblock {\em Port. Math.}, 57(3):311--323, 2000.
\newblock \url{https://www.emis.de/journals/PM/57f3/pm57f305.pdf}.

\bibitem{CPZ2016}
Cristian Cobeli, Mihai Prunescu, and Alexandru Zaharescu.
\newblock A growth model based on the arithmetic {{\(Z\)}}-game.
\newblock {\em Chaos Solitons Fractals}, 91:136--147, 2016.
\newblock \url{https://doi.org/10.1016/j.chaos.2016.05.016}.
\newblock \href {https://doi.org/10.1016/j.chaos.2016.05.016} {\path{doi:10.1016/j.chaos.2016.05.016}}.

\bibitem{CVZ2010}
Cristian Cobeli, Marian V{\^a}j{\^a}itu, and Alexandru Zaharescu.
\newblock On the intervals of a third between {Farey} fractions.
\newblock {\em Bull. Math. Soc. Sci. Math. Roum., Nouv. S{\'e}r.}, 53(3):239--250, 2010.

\bibitem{CZ2013}
Cristian Cobeli and Alexandru Zaharescu.
\newblock Promenade around {Pascal} triangle -- number motives.
\newblock {\em Bull. Math. Soc. Sci. Math. Roum., Nouv. S{\'e}r.}, 56(1):73--98, 2013.
\newblock \url{https://www.jstor.org/stable/43679285}.

\bibitem{CZ2014}
Cristian Cobeli and Alexandru Zaharescu.
\newblock A game with divisors and absolute differences of exponents.
\newblock {\em J. Difference Equ. Appl.}, 20(11):1489--1501, 2014.
\newblock \url{https://doi.org/10.1080/10236198.2014.940337}.
\newblock \href {https://doi.org/10.1080/10236198.2014.940337} {\path{doi:10.1080/10236198.2014.940337}}.

\bibitem{Gil2011}
Norman Gilbreath.
\newblock Processing process: the {Gilbreath} conjecture.
\newblock {\em J. Number Theory}, 131(12):2436--2441, 2011.
\newblock \url{https://doi.org/10.1016/j.jnt.2011.06.008}.
\newblock \href {https://doi.org/10.1016/j.jnt.2011.06.008} {\path{doi:10.1016/j.jnt.2011.06.008}}.

\bibitem{Guy1988}
Richard~K. Guy.
\newblock The strong law of small numbers.
\newblock {\em Am. Math. Mon.}, 95(8):697--712, 1988.
\newblock \url{https://doi.org/10.2307/2322249}.
\newblock \href {https://doi.org/10.2307/2322249} {\path{doi:10.2307/2322249}}.

\bibitem{Guy2004}
Richard~K. Guy.
\newblock {\em Unsolved problems in number theory}.
\newblock Probl. Books Math. New York, NY: Springer-Verlag, 3rd ed. edition, 2004.

\bibitem{KRNG2024}
Jerzy Kaczorowski, Waldemar Ratajczak, Peter Nijkamp, and Krzysztof Górnisiewicz.
\newblock Economic hierarchical spatial systems -- new properties of {Löschian} numbers.
\newblock {\em Applied Mathematics and Computation}, 461:128319, 2024.
\newblock \href {https://doi.org/10.1016/j.amc.2023.128319} {\path{doi:10.1016/j.amc.2023.128319}}.

\bibitem{Leh1956}
D.~H. Lehmer.
\newblock On the {Diophantine} equation {{\(x^3 + y^3 + z^3 = 1\)}}.
\newblock {\em J. Lond. Math. Soc.}, 31:275--280, 1956.
\newblock \url{https://doi.org/10.1112/jlms/s1-31.3.275}.
\newblock \href {https://doi.org/10.1112/jlms/s1-31.3.275} {\path{doi:10.1112/jlms/s1-31.3.275}}.

\bibitem{Loc1940}
August Lösch.
\newblock {\em Economics of location}.
\newblock Yale University Press, 1954.
\newblock {\url{https://archive.org/details/economicsoflocat00ls/page/109}}.

\bibitem{Mah1936}
K.~Mahler.
\newblock Note on hypothesis {{\(K\)}} of {Hardy} and {Littlewood}.
\newblock {\em J. Lond. Math. Soc.}, 11:136--138, 1936.
\newblock \url{https://doi.org/10.1112/jlms/s1-11.2.136}.
\newblock \href {https://doi.org/10.1112/jlms/s1-11.2.136} {\path{doi:10.1112/jlms/s1-11.2.136}}.

\bibitem{MV1975}
H.~L. Montgomery and R.~C. Vaughan.
\newblock The exceptional set in {Goldbach}'s problem.
\newblock {\em Acta Arith.}, 27:353--370, 1975.
\newblock \url{https://doi.org/10.4064/aa-27-1-353-370}.
\newblock \href {https://doi.org/10.4064/aa-27-1-353-370} {\path{doi:10.4064/aa-27-1-353-370}}.

\bibitem{oeis}
{OEIS Foundation Inc\!\!}
\newblock The {O}n-{L}ine {E}ncyclopedia of {I}nteger {S}equences, 2023.
\newblock Published electronically at \url{http://oeis.org}.

\bibitem{PV1991}
G.~Payne and L.~N. Vaserstein.
\newblock Sums of three cubes.
\newblock In {\em The arithmetic of function fields. Proceedings of the workshop at the Ohio State University, June 17-26, 1991, Columbus, Ohio (USA)}, pages 443--454. Berlin: Walter de Gruyter, 1992.

\bibitem{PP1992}
Alberto Perelli and J{\'a}nos Pintz.
\newblock On the exceptional set for the 2{{\(k\)}}-twin primes problem.
\newblock {\em Compos. Math.}, 82(3):355--372, 1992.

\bibitem{Pin2023}
J{\'a}nos Pintz.
\newblock A new explicit formula in the additive theory of primes with applications. {I}: {The} explicit formula for the {Goldbach} problem and the {Generalized} {Twin} {Prime} {Problem}.
\newblock {\em Acta Arith.}, 210:53--94, 2023.
\newblock \url{https://doi.org/10.4064/aa220728-31-3}.
\newblock \href {https://doi.org/10.4064/aa220728-31-3} {\path{doi:10.4064/aa220728-31-3}}.

\bibitem{Pro1878}
F.~Proth.
\newblock Sur la s{\'e}rie des nombres premiers.
\newblock {\em Nouvelle Correspondance Math{\'e}matique}, 4:236--240, 1878.
\newblock \url{https://gdz.sub.uni-goettingen.de/download/pdf/PPN598948236\_0004/LOG\_0088.pdf}.

\bibitem{Pru2022}
Mihai Prunescu.
\newblock Symmetries in the {Pascal} triangle: {{\(p\)}}-adic valuation, sign-reduction modulo~{{\(p\)}} and the last non-zero digit.
\newblock {\em Bull. Math. Soc. Sci. Math. Roum., Nouv. S{\'e}r.}, 65(4):431--447, 2022.
\newblock \url{https://ssmr.ro/bulletin/pdf/65-4/articol_6.pdf}.

\bibitem{VV2013}
Tom Verhoeff and Koos Verhoeff.
\newblock Folded strips of rhombuses and a plea for the $\sqrt{2}:1$ rhombus.
\newblock In George~W. Hart and Reza Sarhangi, editors, {\em Proceedings of Bridges 2013: Mathematics, Music, Art, Architecture, Culture}, pages 71--78, Phoenix, Arizona, 2013. Tessellations Publishing.
\newblock URL: \url{http://archive.bridgesmathart.org/2013/bridges2013-71.html}.

\end{thebibliography}


\end{document}